\documentclass[10pt,a4paper]{amsart}

\usepackage{amssymb, amsmath,amsthm,dsfont}

\usepackage[colorlinks,pdfpagelabels,pdfstartview = FitH,bookmarksopen
= true,bookmarksnumbered = true,linkcolor = blue,plainpages =
false,hypertexnames = false,citecolor = red, pagebackref=false]{hyperref}

\def\N{{\mathds{N}}}
\def\R{{\mathds{R}}}
\def\eps{\varepsilon}

\def\wto{\rightharpoondown}
\newcommand{\wsto}{\overset{\raisebox{-1ex}{\scriptsize $*$}}{\rightharpoondown}}
\def\T{\mbox{\upshape T}_b}
\def\E{E}
\def\HM{\mathcal{H}}
\def\Hm{\HM^{m-1}}
\def\Lm{\mathcal{L}^m}
\def\Hdim{\mbox{\upshape dim}_{\mathcal{H}}}
\def\B{\mathcal{B}_\Lambda}
\def\Bto{\rightrightarrows}
\def\d{\mathrm{d}}

\def\x{{\times}}

\def\edge{\hspace{0.1em}\mbox{\LARGE$\llcorner$}\hspace{0.05em}}

\def\C{\mathcal{C}}
\def\F{\mathds{F}}
\def\CF{\mathcal{F}_i}

\DeclareMathOperator\Div{div}
\DeclareMathOperator\sing{sing}
\DeclareMathOperator\spt{spt}
\DeclareMathOperator\trace{tr}

\newcommand{\xint}[3]{{\setbox0=\hbox{$#1{#2#3}{\int}$}
   \vcenter{\hbox{$#2#3$}}\kern-.5\wd0}}
\newcommand{\mint}{\mathchoice
   {\xint\displaystyle\textstyle-}
   {\xint\textstyle\scriptstyle-}
   {\xint\scriptstyle\scriptscriptstyle-}
   {\xint\scriptscriptstyle\scriptscriptstyle-}
   \!\int}

\newtheorem{thm}{Theorem}[section]
\newtheorem{defn}[thm]{Definition}
\newtheorem{lem}[thm]{Lemma}
\newtheorem{prop}[thm]{Proposition}
\newtheorem{cor}[thm]{Corollary}
\newtheorem{rem}[thm]{Remark}

\begin{document}

\title[Boundary regularity for biharmonic maps]{Blow-up analysis and boundary regularity for variationally biharmonic
  maps}

\author[S. Altuntas]{Serdar Altuntas}
\address{Serdar Altuntas\\ Fakult\"at f\"ur Mathematik, 
Universit\"at Duisburg-Essen\\Thea-Leymann-Stra\ss e 9\\45127 Essen, Germany}
\email{serdar.altuntas@uni-due.de}

\author[C. Scheven]{Christoph Scheven}
\address{Christoph Scheven\\ Fakult\"at f\"ur Mathematik, 
Universit\"at Duisburg-Essen\\Thea-Leymann-Stra\ss e 9\\45127 Essen, Germany}
\email{christoph.scheven@uni-due.de}

\date{\today}
\maketitle

\begin{abstract}
  We consider critical points $u:\Omega\to N$ of the bi-energy
  \begin{equation*}
    \int_\Omega |\Delta u|^2\,\d x,
  \end{equation*}
  where $\Omega\subset\R^m$ is a bounded smooth domain of dimension
  $m\ge 5$
  and $N\subset\R^L$ a compact submanifold without boundary.
  More precisely, we consider 
  variationally biharmonic maps  $u\in W^{2,2}(\Omega,N)$, which are
  defined as critical points of the bi-energy that satisfy
  a certain stationarity condition up to the boundary.
  For weakly convergent sequences of variationally biharmonic maps, we
  demonstrate that the only obstruction that can prevent the
  strong compactness up to the boundary
  is the presence of certain non-constant
  biharmonic $4$-spheres or $4$-halfspheres in the target manifold.
  As an application, we deduce full boundary regularity of
  variationally biharmonic maps provided such spheres do not exist. 
\end{abstract}

\section{Introduction and statement of the results}

Biharmonic maps are a higher order variant of harmonic maps
$u\in C^\infty(\Omega,N)$ into a Riemannian manifold $N\subset\R^L$, which are
defined as critical points of the Dirichlet energy
\begin{equation*}
  E_1(u):=\int_\Omega |Du|^2\d x.
\end{equation*}
Analogously, we call a map $u\in C^\infty(\Omega,N)$ biharmonic if it is a
critical point of the bi-energy
\begin{equation*}
  E_2(u):=\int_\Omega |\Delta u|^2\d x.
\end{equation*}
More generally, maps $u\in W^{2,2}(\Omega,N)$ that satisfy the
Euler-Lagrange equation of $E_2$ in the weak sense are called weakly
biharmonic, and the class of weakly harmonic maps is defined accordingly. 
The analytical and geometric properties of harmonic maps have been
extensively studied over the last decades and are quite well
understood. We refer to \cite{TwoReports} for an overview over the
classical theory.
The theory of biharmonic maps, however, is not yet
developed up to the same level as the one of harmonic maps.
In the present article, we analyse the behaviour of biharmonic maps at
the boundary and investigate the questions of compactness properties
and regularity up to the boundary. 

Before discussing the state of the regularity theory for biharmonic maps,
let us briefly recall some of the main results on regularity of harmonic
maps.  
For minimizing harmonic maps, i.e. minimizers of the Dirichlet energy
in a given Dirichlet class, Schoen \& Uhlenbeck proved that the
singular set can have at most Hausdorff-dimension $m-3$, see
\cite{Schoen-U}. An alternative proof was later given by Luckhaus
\cite{Luckhaus}. Moreover, in \cite{Schoen-U2}, Schoen
\& Uhlenbeck were even able to prove full regularity in a neighbourhood
of the boundary. For harmonic maps that are not minimizing, 
only slightly weaker results are known.
First of all, no regularity results can be derived in super-critical
dimensions $m>2$ for weakly harmonic maps that do not satisfy a
certain energy monotonicity formula,
cf. \cite{Riviere-discont}.
For the slightly smaller class of stationary
harmonic maps, however, Bethuel \cite{Bethuel} established that the singular
set has vanishing $(m-2)$-dimensional Hausdorff measure, see also
\cite{Riviere,Partial-Riviere} for an alternative proof. The reduction
of the dimension to the upper bound $m-3$ as in the case of
minimizers is not known in the general situation. On a technical
level, the reason is
that weakly convergent sequences of stationary harmonic maps may not
have a strongly convergent subsequence \cite[Example 1.1]{Lin},
differently from the case of
minimizing harmonic maps \cite{Schoen-U,Luckhaus}.
Therefore, it is not possible to derive the
dimension bound for the singular set by means of Federer's dimension
reduction principle.
However, a deep result by Lin \cite{Lin} states that this lack of
compactness can occur if and only if the target manifold contains a
non-constant smooth harmonic $2$-sphere $v:S^2\to N$. As a consequence,
under the assumption that the target manifold does not carry any
non-trivial harmonic $2$-spheres, it is possible to prove full
regularity in the neighbourhood of the boundary also for more general
critical points and not only for minimizers \cite{Lin,Scheven0}.

The regularity of biharmonic maps was first investigated by Chang,
Wang \& Yang \cite{CWY}, see also
\cite{Wang_sphere,Wang4d,Wang}, with the result that for any
stationary biharmonic map, the set of interior singular points is
negligible with respect to the $(m-4)$-dimensional Hausdorff
measure. In the case of minimizing biharmonic maps, the dimension of
the interior singular set was further reduced to at most $m-5$ by the second
author \cite{Scheven1}. The latter article also contains
results for stationary biharmonic maps under the assumption that the
target manifold does not carry any non-constant Paneitz-biharmonic
$4$-spheres (cf. Definition \ref{def:paneitz-biharmonic}), which turns
out to be the analogue of the condition found by Lin \cite{Lin} in the
harmonic map case.
As for harmonic maps, an indispensable tool
for all mentioned partial regularity results in
super-critical dimensions $m>4$ is an energy monotonicity formula. For
biharmonic maps, this formula was derived in the interior case
by Angelsberg \cite{Angelsberg}.
In the boundary situation, however, the question
for the corresponding monotonicity formula remained open for some
time. In fact, since such a formula was unknown, the first results on
partial boundary regularity \cite{GongLammWang} and
full boundary regularity for minimizers \cite{Mazowiecka} had to
impose this monotonicity property as an additional assumption.
This gap in the theory has been
closed by the first author \cite{Altuntas}, who provided a suitable boundary
monotonicity formula and thereby completed the mentioned results from
\cite{GongLammWang,Mazowiecka}.

The present article now is concerned with the question whether full
boundary regularity can also be derived for biharmonic maps that are
not minimizing, but only critical points of the bi-energy. The
suitable notion of critical point in the boundary situation
is that of a variationally biharmonic map, see
Definition~\ref{def:variationally-bi}. This notion, which is slightly stronger
than that of a stationary biharmonic map, has been introduced in
\cite{Scheven0} in the harmonic map case and allows in particular to use any 
variation of the domain that keeps the boundary values fixed.
Our first main result is a 
compactness property for sequences of variationally biharmonic
maps. For the proof, we adapt the strategy from \cite{Scheven1}, which are in
turn based on \cite{Lin}, to the boundary
case. The arguments consist in an intricate blow-up analysis of the
defect measure, which detects a possible lack of strong convergence.
We achieve basically the analogous result
as in the harmonic map case, with the only
exception that additionally to the non-existence of
non-trivial Paneitz-biharmonic $4$-spheres, we also need to exclude
the existence of
non-constant Paneitz-biharmonic $4$-halfspheres with constant boundary
values. The reason is that the non-existence proof of harmonic
$2$-halfspheres from \cite{Lemaire} does not seem to carry over
to the higher order case.
However, it seems to be plausible that whenever it
is possible to exclude nontrivial Paneitz-biharmonic $4$-spheres, the
same arguments will also yield the nonexistence of the corresponding
halfspheres. An example of this principle is given in Proposition
\ref{prop:flat-torus}.

This compactness property is the prerequisite for our second main
result, which ensures the full boundary regularity of variationally
biharmonic maps under the assumption that no biharmonic $4$-spheres
and $4$-halfspheres as above exist in the target manifold.
For the proof, we follow Federer's dimension reduction argument and
analyse tangent maps of variationally biharmonic maps in singular
boundary points. For the construction of the tangent maps, it is
crucial to have the strong convergence properties from our first main
result. Since it is possible to show that all tangent maps necessarily
have to be constant, we can deduce that singular boundary points do
not exist.

Next, we specify our assumptions and state
our main results. 

\subsection{Variationally biharmonic maps}

Let $\Omega\subset\R^m$ be a bounded domain of dimension $m\ge5$. We
prescribe Dirichlet boundary data on a boundary part
$\Gamma\subset\partial\Omega$, where the boundary datum is
given in form of a map 
$g\in C^3(\Gamma_\delta,N)$ defined on a neighborhood
$\Gamma_\delta:=\{x\in\overline\Omega\,:\, \mathrm{dist}(x,\Gamma)<\delta\}$
of $\Gamma$.
We consider critical points of the bi-energy 
\begin{equation*}
  E_2(u):=\int_\Omega|\Delta u|^2\d x
  \qquad\mbox{for }u\in W^{2,2}(\Omega,N).
\end{equation*}

The following notion of weakly biharmonic maps can be derived 
by considering variations of the type $u_s(x):=\pi_N(u(x)+s V(x))$ 
with the nearest-point retraction $\pi_N$ onto $N$. 

\begin{defn}
  A map $u\in W^{2,2}(\Omega,N)$ is called weakly biharmonic iff
  \begin{equation*}
    \int_\Omega \Delta u\cdot \Delta V\,d x=0
  \end{equation*}
  holds true for every vector field 
  $V\in W^{2,2}_0\cap L^\infty(\Omega,\R^L)$ that is tangential along
  $u$ in the sense $V(x)\in T_{u(x)}N$ for a.e. $x\in\Omega$. 
\end{defn}

In other words, a weakly biharmonic map is characterized by the fact
that $\Delta^2u\perp T_uN$ holds in the weak sense. This is equivalent to the
differential equation 
\begin{equation}
  \label{weakly-bi}
  \Delta^2u
  =
  \Delta \big(A(u)(Du\otimes Du)\big)
  +2\mathrm{div}\,\big(D(\Pi(u))\cdot\Delta u\big)-\Delta(\Pi(u))\cdot\Delta u
\end{equation}
in the distributional sense, 
where $A(u)$ denotes the second fundamental form of $N\subset\R^L$ and
$\Pi(u(x)):\R^L\to T_{u(x)}N$ is the orthogonal projection onto the
tangent space at $u(x)\in N$. For a
detailed 
proof, we refer to \cite[Prop. 2.2]{Wang}.
Classical solutions $u\in C^4(\Omega,N)$ of \eqref{weakly-bi}
are called biharmonic maps. 

As mentioned above, in super-critical  dimensions $m\ge 5$
a monotonicity formula is
crucial for the derivation of regularity results. Since such a formula
can not be expected to hold for general weakly biharmonic maps, we
have to consider a stronger notion of biharmonicity.
Considering variations of the type $u_s(x)=u(x+s\xi(x))$ leads to the
following notion of biharmonic maps.
\begin{defn}\label{def:variationally-bi}
  A map $u\in W^{2,2}(\Omega,N)$ is called \emph{stationary
    biharmonic} iff it is weakly biharmonic in $\Omega$ and satisfies
  the differential equation 
  \begin{align}\label{stationary-bi-interior}
    &\int_\Omega \big(4\Delta u\cdot D^2uD\xi+2\Delta u\cdot
    Du\,\Delta\xi-|\Delta u|^2\mathrm{div}\,\xi\big)\,\d x=0
  \end{align}
  for every $\xi\in C^\infty_0(\Omega,\R^m)$. 
\end{defn}

However, it turns out that this notion is still not sufficient for the
treatment of the Dirichlet problem, since the differential
equation~\eqref{stationary-bi-interior} only contains information on
interior properties of solutions.
Therefore, we rely on the following notion of biharmonic maps that
is adapted to the Dirichlet boundary problem. 

\begin{defn}
 A map $u\in W^{2,2}(\Omega,N)$ is called \emph{variationally
   biharmonic} with respect to the Dirichlet datum $g$ on $\Gamma$ if
   \begin{equation}\label{Dirichlet-Gamma}
    (u,Du)=(g,Dg)\qquad\mbox{on $\Gamma$ in the sense of traces,}
  \end{equation}
  and if
  \begin{equation*}
    \frac{d}{ds}\bigg|_{s=0}E_2(u_s)=0
  \end{equation*}
  holds true for every variation $u_s\in W^{2,2}(\Omega,N)$,
  $s\in(-\eps,\eps)$, for which the above derivative exists, which
  satisfies $u_0=u$, the boundary condition~\eqref{Dirichlet-Gamma}
  with $u_s$ in place of $u$, and $u_s=u$ a.e. on $\Omega\setminus K$
  for some compact set $K\subset\Omega\cup\Gamma$.  
\end{defn}

\subsubsection{Paneitz-biharmonic maps}

For maps on the $k$-dimensional upper halfsphere $S^k_+$, there are corresponding notions of
biharmonicity. For the purposes of the present article, it suffices to
consider the case $k=4$. 
We prescribe boundary values in form of a map $g\in
C^3(U_\delta,N)$, where we abbreviated
$U_\delta:=S^{4}_+\cap(\R^4\times[0,\delta))$.  For maps defined on
the $4$-sphere $S^4$, we define the Paneitz-bi-energy by 
\begin{equation*}
  P_{S^4}(u):=\int_{S^4}\big[|\Delta_S u|^2+2|Du|^2\big]\d x
  \qquad\mbox{for }u\in W^{2,2}(S^k,N),
\end{equation*}
with the Laplace-Beltrami operator $\Delta_S$ on $S^4$. Analogously,
we define $P_{S^4_+}(u)$ for any $u\in W^{2,2}(S^4_+,N)$.
The Euler-Lagrange operator of this functional is given by the
Paneitz-operator $\mathcal{P}u:=\Delta_S^2u-2\Delta_S u$ on $S^4$, which plays an
important role in comformal geometry. In particular, well-known
properties of the Paneitz operator imply that the Paneitz-bi-energy
$P_{S^4}$ is conformally invariant, cf. \cite{Paneitz,Chang}. 
Critical points of $P_{S^k}$ are called
Paneitz-biharmonic maps in the following sense.
\begin{defn}\label{def:paneitz-biharmonic}
  A map $u\in C^4(S^4,N)$ is called \emph{Paneitz-biharmonic} if 
  \begin{equation}\label{paneitz-biharmonic}
    \Delta_S^2u(x)-2\Delta_Su(x)\perp T_{u(x)}N
  \end{equation}
  holds true for any $x\in S^4$.
  Analogously, a map $u\in C^4(S^4_+,N)$ is called
  \emph{Paneitz-biharmonic}
  with Dirichlet datum $g$ if $(u,Du)=(g,Dg)$ on $\partial S^4_+$ and
  \eqref{paneitz-biharmonic} holds true
  for any $x\in S^4_+$.
\end{defn}

\begin{defn}
  We say that the Riemannian manifold $N$ does not carry any
  non-constant Paneitz-biharmonic $4$-spheres if every
  Paneitz-biharmonic map $u\in C^4(S^4,N)$ is constant.

  Analogously, we say that $N$ does not carry any non-constant Paneitz-biharmonic
  $4$-halfspheres with constant boundary values if every Paneitz-biharmonic map $u\in C^4(S^4_+,N)$
  with $(u,Du)=(c,0)$ on $\partial S^4_+$ for some constant $c\in N$
  satisfies $u=c$ on $S^4_+$. 
\end{defn}

\subsection{Statement of the results}

Now we are in a position to state our main results. In all statements,
we restrict
ourselves to the case of a flat boundary, i.e. to the case that
$\Omega$ is a half ball $B_R^+$, with boundary values prescribed on the flat
part of the boundary, which we denote by $T_R$. The general case of a smooth boundary can be
reduced to this case by flattening the boundary. However, this
procedure will change the Euclidean metric to a more general
Riemannian one. Nevertheless, we decided to treat only the model case
of the Euclidean metric in order not to
overburden this work with additional technicalities.

Our first main result is the following compactness property for
bounded sequences of variationally biharmonic maps. 

\begin{thm}\label{thm:compact}
  Let $N\subset\R^L$ be a compact, smooth Riemannian manifold that
  does neither carry non-constant Paneitz-biharmonic $4$-spheres nor
  non-constant Paneitz-biharmonic $4$-halfspheres with constant boundary values. 
  Assume that $g_i\in C^\infty(B_4^+,N)$, $i\in\N$, is a sequence of boundary
  values and $u_i\in W^{2,2}(B_4^+,N)$ are
  variationally biharmonic maps with respect to the Dirichlet data
  $g_i$ on $T_4$, and that both sequences are bounded in the sense   
  \begin{equation*}
    \sup_{i\in\N}\,\|u_i\|_{W^{2,2}(B_4^+)}<\infty
    \qquad\mbox{and}\qquad
    \sup_{i\in\N}\,\|g_i\|_{C^{4,\alpha}(B_4^+)}<\infty,
  \end{equation*}
  for some $\alpha\in(0,1)$. 
  Then there is a map $u\in  W^{2,2}(B_1^+,N)$ so that after passing to
  a subsequence, we have the convergence 
  \begin{equation*}
    u_{i}\to u \qquad\mbox{in $W^{2,2}(B_{1}^+,\R^L)$, as $i\to\infty$.} 
  \end{equation*}
\end{thm}

\begin{rem}
  \textup{The assumption on the non-existence of Paneitz-biharmonic
    $4$-spheres and $4$-halfspheres is necessary in the following
    sense. Assume that there is a non-constant Paneitz-biharmonic
    $4$-halfsphere with constant boundary values. Then, by means of
    stereographic projection and the conformal invariance of the
    Paneitz-bienergy, we infer a non-constant biharmonic map $u\in
    C^\infty(\R^4_+,N)$ of finite bi-energy
    with $(u,Du)=(c,0)$ on $\partial\R^4_+$. This
    map gives rise to the sequence of rescaled biharmonic maps
    \begin{equation*}
      u_i\in C^\infty(B_1^{4,+},N),\quad
      u_i(x):=u(\lambda_ix)
      \ \mbox{for }x\in B_1^{4,+}\subset\R^4_+\,,
 \end{equation*}
for any sequence $\lambda_i\to\infty$, with the abbreviation $B_r^{4,+}$ for the
four-dimensional upper halfball of radius $r>0$.
For this sequence of biharmonic maps, we have 
\begin{equation*}
  \sup_{i\in\N}\int_{B_1^{4,+}}|\Delta u_i|^2\,dx
  =
  \sup_{i\in\N}\int_{B_{\lambda_i}^{4,+}}|\Delta u|^2\,dx
  =
  \int_{\R^4_+}|\Delta u|^2\,dx<\infty,
\end{equation*}
which implies $\sup_{i\in\N}\|u_i\|_{W^{2,2}(B_1^{4,+})}<\infty$ by
standard $L^2$-estimates for the Laplace operator.
Moreover, for an arbitrary map 
$\varphi\in C^0(\overline{B_1^{4,+}})$ with compact support, we compute
\begin{equation*}
  \int_{B_1^{4,+}}|\Delta u_i(x)|^2\varphi(x)\,\d x
  =
  \int_{B_{\lambda_i}^{4,+}}|\Delta u(y)|^2\varphi(\lambda_i^{-1}y)\,\d y
  \to
  \varphi(0)\int_{\R^4_+}|\Delta u|^2\,\d y,
\end{equation*}
which corresponds to the convergence of measures $\Lm\edge|\Delta
v_i|^2\wto c\delta_0$ in the limit $i\to\infty$, with the Dirac
measure $\delta_0$ and the constant $c=\int_{\R^4_+}|\Delta u|^2\d
y>0$. The latter convergence contradicts the subconvergence of the sequence
$u_i$ with respect to the $W^{2,2}$-norm.\\
Analogously, the existence of a non-constant Paneitz-biharmonic
$4$-sphere yields a sequence of biharmonic maps $u_i$ on the full ball
$B_1^4\subset\R^4$ that does not contain a strongly convergent subsequence.
This demonstrates the necessity of the assumptions on
Paneitz-biharmonic spheres and halfspheres for our compactness result.
}
\end{rem}

The preceding compactness result is the crucial step for the derivation of the
full boundary regularity for variationally biharmonic maps.

\begin{thm}\label{thm:boundary-regularity}
  Let $N\subset\R^L$ be a smooth compact Riemannian manifold that does
  neither carry a non-constant Paneitz-biharmonic $4$-sphere nor a
  non-constant Paneitz-biharmonic $4$-halfsphere with constant
  boundary values. 
  Assume that the map $u\in W^{2,2}(B_1^+,N)$ is variationally biharmonic
  with $(u,Du)=(g,Dg)$ on $T_1$ in the sense of traces,
  for Dirichlet values $g\in
  C^\infty(B_1^+,N)$. Then $u$ is smooth in a full neighbourhood of
  $T_1$. 
\end{thm}

\subsection{Plan of the paper}
In Section~\ref{sec:preliminaries}, we gather some
technical tools that will be crucial for our arguments. In
particular, we derive a Morrey space estimate that is a consequence of
a boundary monotonicity formula, and we recall some partial regularity
results for variationally biharmonic maps. Moreover, we prove a
gradient estimate under a smallness assumption on the energy.

Section~\ref{sec:compact} is then devoted to the proof of our
compactness result. We introduce the notion of defect measure for a
sequence of variationally biharmonic maps and analyse in what sense
this measure detects the lack of strong convergence. We show that if a
nontrivial defect measure exists, then a blow-up
procedure yields a flat defect measure that is supported on an
$(m-4)$-dimensional plane. This means that we can find a sequence of
variationally biharmonic maps that converge strongly away from this
plane. In this more controlled situation, it is then possible to
follow the ideas by Lin \cite{Lin} and to construct a suitable
blow-up sequence around carefully chosen blow-up points. Depending on
whether these points approach the boundary or not, we can show that the
limit gives rise to a non-constant Paneitz-biharmonic $4$-halfsphere
with constant boundary values or to a corresponding full
$4$-sphere. Since the existence of such maps is excluded by
assumption, we deduce the desired strong compactness for any bounded
sequence of variationally biharmonic maps.

The next Section~\ref{sec:liouville} contains some Liouville type
results for biharmonic maps on half-spaces that arise as tangent maps
of biharmonic maps in boundary points. Since we can show that all
possible tangent maps are constant, 
the implementation of Federer's dimension reduction principle
in Section~\ref{sec:dim-red} allows us to deduce our second main result on the
full boundary regularity of variationally biharmonic maps.\\ 

\noindent
\textbf{Acknowledgments. }This work has been supported by the
DFG-project SCHE 1949/1-1 ``Randregularit\"at biharmonischer Abbildungen
zwischen Riemann'schen Mannigfaltigkeiten''.

\section{Preliminaries}
\label{sec:preliminaries}

\subsection{Notation}

We write $B_r(a)\subset\R^m$ for the open ball with radius $r>0$ and
center $a\in\R^m$ and $S_r(a):=\partial B_r(a)$ for the corresponding
sphere.
For the upper halfspace, we use the abbreviation
$\R^m_+:=\R^{m-1}\times[0,\infty)$, and
write $B_r^+(a):=B_r(a)\cap\R^m_+$ for arbitrary centers $a\in\R^m_+$.
Moreover, in the case
of a center $a=(a',0)\in\partial\R^m_+$ we use the
abbreviations
$S_r^+(a):=S_r(a)\cap\R^m_+$ for the curved part and 
$T_r(a):=B_r(a')\times\{0\}$ for the flat part of the
boundary of $B_r^+(a)$.  If the center is clear
from the context, we will often omit it in the notation and simply
write $B_r$, $B_r^+$, $S_r$, $S_r^+$, $T_r$ instead of $B_r(a)$, $B_r^+(a)$,
$S_r(a)$, $S_r^+(a)$, $T_r(a)$.

For the Lebesgue measure on $\R^m$ we write $\mathcal{L}^m$, and
for $0\le k\le m$, the $k$-dimensional Hausdorff measure on $\R^m$
will be abbreviated by $\mathcal{H}^k$.

For $f\in L^2(B_2^+)$ and $\lambda>0$, we write 
\begin{equation*}
  \|f\|_{L^{2,\lambda}(B_2^+)}^2
  :=
  \sup_{B_\rho^+(y)\subset
    B_2^+}\frac1{\rho^\lambda}\int_{B_\rho^+(y)}|f|^2\,\d x.
\end{equation*}

Finally, the singular set of a map $u:\Omega\to\R^L$ is defined by
\begin{equation*}
  \sing(u):=\overline\Omega\setminus \left\{x\in\overline\Omega\,:\, u\in
    C^\infty(B_\rho(x)\cap\overline\Omega,\R^L)\mbox{ for some }\rho>0 \right\}.
\end{equation*}

\subsection{Monotonicity formula and consequences}

The proof of the following lemma can be retrieved from \cite[Lemma
2.1]{Altuntas}. 

\begin{lem}
  Let $u\in W^{2,2}(\Omega,N)$ be a variationally biharmonic map
  with respect to the Dirichlet datum $g$ on
  $\Gamma\subset\partial\Omega$. Then for every $\xi\in
  C^\infty(\Omega\cup\Gamma,\R^m)$ with $\xi(x)\in T_x(\partial\Omega)$
  for every $x\in\Gamma$ and $\spt\xi\Subset\Omega\cup\Gamma$, we have
  \begin{align}\label{stationary-bi}
    &\int_\Omega \big(4\Delta u\cdot D^2uD\xi+2\Delta u\cdot
    Du\,\Delta\xi-|\Delta u|^2\mathrm{div}\,\xi\big)\,\d x\\\nonumber
    &\quad=
    \int_\Omega 2\Delta u\cdot \Delta \big[\Pi(u)(Dg\,\xi)\big]\,\d x,
  \end{align}
  with the orthogonal projection $\Pi(u(x)):\R^L\to T_{u(x)}N$. 
\end{lem}

The preceding lemma is the first step in the derivation of a boundary
monotonicity formula, which has been proven in 
\cite{Altuntas} for the case of a flat boundary.
More precisely, we specialize to the case 
$\Omega=B_R^+(a)$ and $\Gamma=T_R(a)$.

\begin{thm}{\upshape (\cite[Theorem 1.2]{Altuntas})}\label{thm:monotonicity}
  Let $m\ge 5$, $0<R\le1$
  and assume that
  $u\in W^{2,2}(B_R^+(a),N)$ satisfies
  \eqref{stationary-bi}. Then, for a.e. radii
  $0<\rho<r<R$, we have the monotonicity formula
  \begin{align}\label{bdry-monotonicity}
    &\Phi_u(a;\rho)
    +
    4\int_{B_r^+(a)\setminus B_\rho^+(a)}e^{\chi|x-a|}\bigg(\frac{|D\partial_Xu|^2}{|x-a|^{m-2}}+(m-2)\frac{|\partial_Xu|^2}{|x-a|^m}\bigg)\,\d x\\\nonumber
    &\quad\le
    \Phi_u(a;r)
    +
    K\Psi_u(a;\rho,r),    
  \end{align}
  where we abbreviated
  \begin{align}\label{Def-Phi}
    \Phi_u(a;r)&:=e^{\chi r}r^{4-m}\int_{B_r^+(a)}|\Delta u|^2\,\d x\\\nonumber
    &\qquad+ e^{\chi
      r}r^{3-m}\int_{S_r^+(a)}\big(\partial_X|Du|^2+4|Du|^2-4r^{-2}|\partial_Xu|^2\big)\,d\mathcal{H}^{m-1},
  \end{align}
  with the short-hand notation $\partial_X:=(x_i-a_i)\partial_i$, and
  \begin{align}\label{Def-Psi}
    \Psi_u(a;\rho,r)
    &:=r+\int_{B_r^+(a)\setminus
      B_\rho^+(a)}\bigg(\frac{|D^2u|^2}{|x-a|^{m-5}}+\frac{|Du|^2}{|x-a|^{m-3}}\bigg)\,\d x\\\nonumber
    &\qquad+
    \int_{S_r^+(a)\cup S_\rho^+(a)}\frac{|D^2u|^2}{|x-a|^{m-6}}\,d\mathcal{H}^{m-1}.
  \end{align}
  In the above formula, the constants $\chi,K\ge0$ depend only on the data
  $m,N$ and $\|Dg\|_{C^2}$. In particular, the constants
  $\chi$ and $K$ vanish in the limit $\|Dg\|_{C^2}\to0$. 
\end{thm}

An important consequence of the preceding monotonicity formula is the
following Morrey space estimate.

\begin{lem}\label{lemma:morrey}
  Let $m\ge 5$ and $g\in C^3(B_1^+,N)$ be given.
  Assume that $u\in W^{2,2}(B_1^+,N)$ satisfies
  \eqref{stationary-bi} and that $(u,Du)=(g,Dg)$ holds on $T_1$, in
  the sense of traces. Then for any ball
  $B_R^+(a)\subset B_1^+$ with radius $R\in(0,1]$ we have 
  \begin{align*}
    \sup_{B_\rho^+(y)\subset B_{R/2}^+(a)}\,&
    \rho^{4-m}\int_{B_\rho^+(y)}\big(|D^2u|^2+\rho^{-2}|Du|^2\big)\,\d x\\
    &\le
    c_1R^{4-m}\int_{B_R^+(a)}\big(|\Delta u|^2+R^{-2}|Du|^2\big)\,\d x
    +c_2R,
  \end{align*}
  with constants $c_1$, $c_2$ that depend at most on $m,N,$ and
  $\|Dg\|_{C^2}$. Moreover, we have $c_2\to0$ in the limit
  $\|Dg\|_{C^2}\to0$.  
\end{lem}

\begin{proof}
  For the proof, we modify some ideas from \cite{Struwe} and
  \cite{Moser}.
  Throughout this proof, we write $c$ for a constant that may depend
  on $m$, $N$, and $\|Dg\|_{C^2}$. We first consider the case $a\in T_1$.
  For $0<s<\frac{R}{8}$, we consider radii
  $\rho\in[2s,4s]$ and $r\in[\frac R2,R]$ that will be chosen later. 
  Our goal is to derive an estimate for the term
  \begin{equation*}
    \widetilde\Phi_u(a;\rho)
    :=
    e^{\chi\rho}\rho^{4-m}\int_{B_\rho^+(a)}|\Delta u|^2\,\d x
    +
    e^{\chi\rho}\rho^{3-m}\int_{S_\rho^+(a)}|Du|^2\,\d \HM^{m-1},
  \end{equation*}
  where $\chi\ge0$ is the constant from
  Theorem~\ref{thm:monotonicity}, which depends only on $m,N,$ and $\|Dg\|_{C^2}$.
  For the difference of this
  term and the term $\Phi_u(a;\rho)$ defined in~\eqref{Def-Phi}, we
  compute, using again the abbreviation $X(x):=x-a$,
  \begin{align*}
    &\widetilde\Phi_u(a;\rho)-\Phi_u(a;\rho)\\
    &\qquad=
    e^{\chi\rho}\rho^{3-m}
    \int_{S_\rho^+(a)}\big(4\rho^{-2}|\partial_Xu|^2
    -\partial_X|Du|^2-3|Du|^2\big)\d\mathcal{H}^{m-1} \\
    &\qquad=
    e^{\chi\rho}\rho^{3-m}
    \int_{S_\rho^+(a)}\big(4\rho^{-2}|\partial_Xu|^2
    -2D(\partial_Xu)\cdot Du-|Du|^2\big)\d\mathcal{H}^{m-1} \\
    &\qquad\le
    e^{\chi\rho}\rho^{3-m}
    \int_{S_\rho^+(a)}\big(4\rho^{-2}|\partial_Xu|^2+|D\partial_Xu|^2\big)\d\mathcal{H}^{m-1},
  \end{align*}
  where we applied Young's inequality in the last step. 
  Taking the mean integral over $\rho\in[2s,4s]$, we deduce 
  \begin{align*}
    &\mint_{2s}^{4s}\widetilde\Phi_u(a;\rho)\d\rho\\
    &\ \le
    \mint_{2s}^{4s}\Phi_u(a;\rho)\d\rho\\
    &\quad\ +
    c\int_{B_{4s}^+(a)\setminus
      B_{2s}^+(a)}e^{\chi|x-a|}\bigg(\frac{|D\partial_Xu|^2}{|x-a|^{m-2}}+(m-2)\frac{|\partial_Xu|^2}{|x-a|^m}\bigg)\d
    x\\
    &\ \le
    \mint_{2s}^{4s}\Phi_u(a;\rho)\d\rho\\
    &\quad\ +
    c\,\mint_{2s}^{4s}\int_{B_{r}^+(a)\setminus
      B_{\rho/2}^+(a)}e^{\chi|x-a|}\bigg(\frac{|D\partial_Xu|^2}{|x-a|^{m-2}}+(m-2)\frac{|\partial_Xu|^2}{|x-a|^m}\bigg)\d
    x\d \rho,
  \end{align*}
  where we used $\frac\rho2\le 2s$ and $4s\le\frac R2\le r$ in the
  last step.
  Both terms on the right-hand side can be estimated by an application
  of the monotonicity formula~\eqref{bdry-monotonicity}. 
  This leads to the estimate
  \begin{align*}
    \mint_{2s}^{4s}\widetilde\Phi_u(a;\rho)\d\rho
    &\le
    c\,\Phi_u(a;r)\\
    &\quad+
    cK\mint_{2s}^{4s}\big(\Psi_u(a;\rho,r)+\Psi_u(a;\tfrac\rho2,r)\big)\d\rho,
  \end{align*}
  for a.e. $r\in[\frac R2,R]$, with the constant
  $K=K(m,N,\|Dg\|_{C^2})$ from Theorem~\ref{thm:monotonicity}. We
  recall that $K\to0$ as $\|Dg\|_{C^2}\to0$. 
  By definition of $\Psi_u$, the last integral can be estimated by
  \begin{align*}
    &\mint_{2s}^{4s}\big(\Psi_u(a;\rho,r)+\Psi_u(a;\tfrac\rho2,r)\big)\d\rho\\
    &\qquad\le
    2r+c\int_{B_r^+(a)\setminus
    B_{s}^+(a)}\bigg(\frac{|D^2u|^2}{|x-a|^{m-5}}+\frac{|Du|^2}{|x-a|^{m-3}}\bigg)\,\d x\\\nonumber
    &\qquad\quad+
    2\int_{S_r^+(a)}\frac{|D^2u|^2}{|x-a|^{m-6}}\,d\mathcal{H}^{m-1}.
  \end{align*}
Integrating by parts twice, we estimate the second last integral as
follows. 
\begin{align}\label{169}
&\int_{B_r^+(a)\setminus B_s^+(a)}\left(\dfrac{\vert D^2u\vert^2}{\vert
  x-a\vert^{m-5}}+\dfrac{\vert Du\vert^2}{\vert
  x-a\vert^{m-3}}\right)\d x\\\nonumber
&\quad=\int_{s}^r
\left(\sigma^{5-m}\int_{S_{\sigma}^+(a)}\vert D^2u\vert^2\d \HM^{m-1}
 +\sigma^{3-m}\int_{S_{\sigma}^+(a)}\vert Du\vert^2\d \HM^{m-1}\right)
  \d \sigma\\\nonumber
  &\quad\le
  r^{5-m}\int_{B_{r}^+(a)}\big(\vert
D^2u\vert^2+r^{-2}|Du|^2\big)\d x\\\nonumber
&\quad\quad+c\int_{s}^r\sigma^{4-m}\int_{B_{\sigma}^+(a)}\big(\vert D^2u\vert^2+\sigma^{-2}|Du|^2\big)\d x\d\sigma\\\nonumber
&\quad\leq
  cr\sup_{\sigma\in[s,r]}\E(\sigma),
\end{align}
where we introduced the notation 
\begin{align*}
  \E(\sigma)
  :=
  \sigma^{4-m}\int_{B_{\sigma}^+(a)}\big(\vert
  D^2u\vert^2+\sigma^{-2}\vert Du\vert^2\big)\d x.
\end{align*}
Combining the three preceding estimates and recalling the definition
of $\Phi_u$, we deduce
\begin{align*}
  &\mint_{2s}^{4s}\widetilde\Phi_u(a;\rho)\d\rho\\\nonumber
  &\qquad\le
  c\,\Phi_u(a;r) + cr^{6-m}\int_{S_r^+(a)}|D^2u|^2\,d\mathcal{H}^{m-1}
  +cKR+cR\sup_{\sigma\in[s,r]}\E(\sigma)\\\nonumber
  &\qquad\le
  cr^{4-m}\int_{B_r^+(a)}|\Delta u|^2\d x+cr^{5-m}\int_{S_r^+(a)}\big(|D^2u|^2+r^{-2}|Du|^2\big)\,d\mathcal{H}^{m-1}\\\nonumber
  &\quad\qquad
   +cKR+
  cR\sup_{\sigma\in[s,r]}\E(\sigma),\nonumber
\end{align*}
for a.e. $r\in[\frac R2,R]$. 
Now we choose a good radius $r\in[\frac R2,R]$ in the sense that
\begin{align*}\nonumber
  r^{5-m}\int_{S_{r}^+(a)}\big(\vert D^2u\vert^2+r^{-2}\vert
  Du\vert^2\big)\d\mathcal{H}^{m-1}
  \leq
  c\E(R).
\end{align*}
Moreover, this radius can be chosen in such a way that the preceding estimate is
valid for this choice of $r$, which implies
\begin{align}\label{estimate-of-the-mean}
  \mint_{2s}^{4s}\widetilde\Phi_u(a;\rho)\d\rho
  \le
  c\E(R) +cKR
  +
  cR\sup_{\sigma\in[s,R]}\E(\sigma).
\end{align}
Our next aim is to estimate $\widetilde\Phi_u(a;\rho)$ from below. First, we
observe that with a standard cut-off function $\eta\in
C^\infty_0(B_\rho(a))$ with $\eta\equiv1$ in $B_{\rho/2}(a)$, two
integrations by parts and Young's inequality lead to the estimate 
\begin{align*}
  &\rho^{4-m}\int_{B_{\rho/2}^+(a)}|D^2(u-g)|^2\d x
  \le
  \rho^{4-m}\int_{B_{\rho}^+(a)}\eta^2|D^2(u-g)|^2\d x\\
  &\qquad\le
  c\rho^{4-m}\int_{B_\rho^+(a)}\big(|\Delta (u-g)|^2+\rho^{-2}|D(u-g)|^2\big)\d x,
\end{align*}
which implies
\begin{align}\label{L2-estimate}
  &\rho^{4-m}\int_{B_{\rho/2}^+(a)}|D^2u|^2\d x\\\nonumber
  &\qquad\le
  c\rho^{4-m}\int_{B_\rho^+(a)}\big(|\Delta
  u|^2+\rho^{-2}|Du|^2\big)\d x
  +c\rho^2\|Dg\|_{C^1}^2,
\end{align}
where $c=c(m)$.
Two applications of Gau\ss' theorem yield
\begin{align*}
  &m\int_{B_\rho^+(a)}|Du|^2\d x\\
  &\quad=
  \int_{B_\rho^+(a)}\mathrm{div}\big((x-a)|Du|^2\big)\d x
  -
  2\int_{B_\rho^+(a)}\partial_XDu\cdot Du\,\d x\\
  &\quad=
  \rho\int_{S_\rho^+(a)}|Du|^2\d\mathcal{H}^{m-1}
  +
  2\int_{B_\rho^+(a)}\big(\partial_Xu\cdot \Delta
  u+|Du|^2\big)\d x\\
  &\quad\qquad
  -2\rho^{-1}\int_{S_\rho^+(a)}|\partial_Xu|^2\,\d\mathcal{H}^{m-1}
  +2\int_{T_\rho(a)}\partial_Xg\cdot\partial_mg\,\d x\\
  &\quad\le
  \rho\int_{S_\rho^+(a)}|Du|^2\d\mathcal{H}^{m-1}\\
  &\quad\qquad+
  \int_{B_\rho^+(a)}\big(\rho^2|\Delta u|^2+3|D u|^2\big)\d x
  +c\rho^{m}\|Dg\|_{C^0}^2. 
\end{align*}
Since $m>3$, we can re-absorb the integral of $|Du|^2$ into the left-hand
side. Multiplying the resulting estimate by $\rho^{2-m}$, we infer 
\begin{align}\label{Du-estimate}
  &\rho^{2-m}\int_{B_\rho^+(a)}|Du|^2\d x\\\nonumber
  &\quad\le
  c\rho^{3-m}\int_{S_\rho^+(a)}|Du|^2\d\mathcal{H}^{m-1}
  +
  c\rho^{4-m}\int_{B_\rho^+(a)}|\Delta u|^2\d x
  +c\rho^2\|Dg\|_{C^0}^2\\\nonumber
  &\quad\le
  c\,\widetilde\Phi_u(a;\rho) +c\rho^2\|Dg\|_{C^0}^2.
\end{align}
Combining estimates \eqref{L2-estimate} and \eqref{Du-estimate}, we
deduce
\begin{align*}
  &\rho^{4-m}\int_{B_{\rho/2}^+(a)}\big(|D^2u|^2+\rho^{-2}|Du|^2\big)\d
  x\\
  &\qquad\le
  c\rho^{4-m}\int_{B_\rho^+(a)}\big(|\Delta
  u|^2+\rho^{-2}|Du|^2\big)\d x
  +c\rho^2\|Dg\|_{C^1}^2\\
  &\qquad\le
  c\,\widetilde\Phi_u(a;\rho) +c\rho^2\|Dg\|_{C^1}^2
\end{align*}
for any $\rho\in[2s,4s]$, which implies in turn 
\begin{align*}
  E(s)
  \le
  c\,\widetilde\Phi_u(a;\rho)+cR^2\|Dg\|_{C^1}^2.
\end{align*}
We use this to estimate the left-hand side of
\eqref{estimate-of-the-mean} from below, with the result
\begin{align*}
  \E(s)&\le
  c\E(R)+c\widetilde K R
   +
  cR\sup_{\sigma\in[s,R]}\E(\sigma),\nonumber
\end{align*}
for every $s\in (0,\frac R8)$, where we abbreviated $\widetilde K:=K+\|Dg\|_{C^1}^2$. We take the supremum over
$s\in [\delta,\frac R8)$ on both sides, for some $\delta>0$, and infer 
\begin{align*}
  \sup_{s\in[\delta,R]}\E(s)
  &\le
  \sup_{s\in[\delta,R/8)}\E(s)
  +c\E(R)\\
  &\leq
  c\E(R)+c\widetilde K R
  +
  cR \sup_{s\in[\delta,R]}\E(s).
\end{align*}
Now we choose a radius $R_0=R_0(m,N,\|Dg\|_{C^2})>0$ so small that $cR_0\le
\frac12$, which allows us to re-absorb the last term into the
left-hand side, provided $R\le R_0$. Letting $\delta\downarrow0$, we
deduce 
\begin{align}\label{Morrey-boundary}
  &s^{4-m}\int_{B_s^+(a)}\big(\vert
  D^2u\vert^2+s^{-2}|Du|^2\big)\d x\\\nonumber
  &\qquad\le
  cR^{4-m}\int_{B_R^+(a)}\big(\vert
  D^2u\vert^2+R^{-2}|Du|^2\big)\d x + c \widetilde KR
\end{align}
for all $s\le R\le R_0$, in the case $a\in T_1$.
For radii $s,R$ with $s\le R_0\le R\le1$, we obtain the same estimate by
applying~\eqref{Morrey-boundary} with $R_0$ in place of $R$ and
then enlarging the domain of integration on the right-hand side. 
Finally, the estimate~\eqref{Morrey-boundary} is immediate in the case
$R_0\le s\le R$, since $R_0$ is a universal constant.
Hence, we obtain~\eqref{Morrey-boundary} for any $s\le
R\le 1$ and boundary points $a\in T_1$.

In the interior case
$B_R(a)\subset B^+$, we can argue in the same way, starting from the
interior version of the monotonicity formula from
\cite{Angelsberg}, to derive the estimate
\begin{align}\label{Morrey-interior}
  &s^{4-m}\int_{B_s(a)}\big(\vert
  D^2u\vert^2+s^{-2}|Du|^2\big)\d x\\\nonumber
  &\qquad\le
  cR^{4-m}\int_{B_R(a)}\big(\vert
  D^2u\vert^2+R^{-2}|Du|^2\big)\d x
\end{align}
for all $s\le R$. The two preceding estimates can be combined in a
standard way to obtain the result. In fact, let
$B_\rho^+(y)\subset B_{R/2}^+(a)$ be an arbitrary ball with
$\rho\le \frac R4$ and $y_m\le\frac R4$. We use the notation $(y',y_m):=y\in\R^{m-1}\times\R$
and let $R_1:=\max\{y_m,\rho\}$. Then we use first the interior
estimate \eqref{Morrey-interior} and then the boundary version
\eqref{Morrey-boundary} to deduce
\begin{align*}
  &\rho^{4-m}\int_{B_\rho^+(y)}\big(\vert
  D^2u\vert^2+\rho^{-2}|Du|^2\big)\d x\\
  &\qquad\le
  cR_1^{4-m}\int_{B_{R_1}^+(y)}\big(\vert
  D^2u\vert^2+R_1^{-2}|Du|^2\big)\d x\\
  &\qquad\le
  cR_1^{4-m}\int_{B_{2R_1}^+(y',0)}\big(\vert
  D^2u\vert^2+R_1^{-2}|Du|^2\big)\d x\\
  &\qquad\le
  cR^{4-m}\int_{B_{R/2}^+(y',0)}\big(\vert
  D^2u\vert^2+R^{-2}|Du|^2\big)\d x +c\widetilde K R\\
  &\qquad\le
  cR^{4-m}\int_{B_{R}^+(a)}\big(\vert
  D^2u\vert^2+R^{-2}|Du|^2\big)\d x+c\widetilde K R.
\end{align*}
In the remaining case $y_m>\frac R4$, the corresponding result follows from the
interior estimate \eqref{Morrey-interior}. Finally, for radii
$\rho>\frac R4$, the above estimate is trivial. We note that
$\widetilde K\to0$ in the limit $\|Dg\|_{C^2}\to0$. Hence, we have
established the assertion in any case. 
\end{proof}

For later reference, we state another consequence of the monotonicity formula.
\begin{cor}\label{cor:density}
  Assume that $u\in W^{2,2}(B_1^+,N)$ satisfies
  \eqref{stationary-bi} and $(u,Du)=(g,Dg)$ on $T_1$, in
  the sense of traces. Then the limit
  \begin{equation*}
    \lim_{\rho\downarrow0}\mint_{\rho/2}^\rho \Phi_u(a;\sigma)\d\sigma
  \end{equation*}
  exists for any $a\in T_1$.
\end{cor}

\begin{proof}
  For two radii $0<\rho<\frac R2<\frac14(1-|a|)$, the monotonicity formula from
  Theorem~\ref{thm:monotonicity} implies
  \begin{equation}\label{averaged-monotonicity}
    \mint_{\rho/2}^\rho \Phi_u(a;\sigma)\d\sigma
    \le
    \mint_{R/2}^{R} \Phi_u(a;s)\d s
    +
    K\mint_{\rho/2}^\rho\mint_{R/2}^{R} \Psi_u(a;\sigma,s)\d s\d\sigma.
  \end{equation}
  Using an integration by parts argument similarly to \eqref{169} and
  then applying Lemma~\ref{lemma:morrey}, we can estimate
  \begin{align*}
    &\mint_{\rho/2}^\rho\mint_{R/2}^R \Psi_u(a;\sigma,s)\d s\d\sigma\\
    &\qquad\le
    cR\sup_{\sigma\in[\rho/2,R]}\sigma^{4-m}\int_{B_\sigma^+(a)}\big(|D
    ^2u|^2+\sigma^{-2}|Du|^2\big)\d x
    \le
    c(u) R.
  \end{align*}
  Here, $c(u)$ denotes a constant that
  depends on $m,N,\|Dg\|_{C^2}$, and $\|u\|_{W^{2,2}}$. We use this to
  estimate the right-hand side of~\eqref{averaged-monotonicity}. Then,
  we first let $\rho\downarrow0$ and then $R\downarrow0$ in the
  resulting estimate, which implies
  \begin{equation*}
    \limsup_{\rho\downarrow0}\mint_{\rho/2}^\rho \Phi_u(a;\sigma)\,\d\sigma
    \le
    \liminf_{R\downarrow0}\mint_{R/2}^R \Phi_u(a;s)\,\d s.
  \end{equation*}
  This concludes the proof of the corollary.
\end{proof}

\subsection{Partial regularity for variationally biharmonic maps}

The following $\eps$-regularity result was first established in
\cite[Thm. 1.1]{GongLammWang} under the additional assumption that a
boundary monotonicity inequality of the
type~\eqref{bdry-monotonicity} is satisfied. The proof was later
completed in \cite{Altuntas}, where the boundary monotonicity formula
was proved for arbitrary variationally biharmonic maps. The Morrey
space estimate that results from the monotonicity formula is stated in
Lemma~\ref{lemma:morrey}. Combining
\cite[Lemma 3.1]{GongLammWang} with Lemma~\ref{lemma:morrey} leads to
the following regularity result.
\begin{thm}\label{epsreg}
  Let $m\ge 5$, $a\in\partial\R^m_+$, $\rho>0$, and $g\in
  C^\infty(B_\rho^+(a),N)$ be given.
  There exists a constant $\eps_1>0$, depending only on $m,N$, and
  $\|Dg\|_{C^2}>0$ so that for every weakly 
  biharmonic map $u\in W^{2,2}(B_\rho^+(a),N)$ with
  \eqref{stationary-bi} that attains the
  Dirichlet datum $g$ on $T_\rho(a)$ and fulfills the estimate 
  \begin{equation*}
    \rho^{4-m}\int_{B_\rho^+(a)}\big(|\Delta u|^2+\rho^{-2}|Du|^2\big)\,\d x+\rho<\eps_1
  \end{equation*}
  we have $u\in C^\infty(B_{\rho/2}^+(a),N)$.
\end{thm}

In a standard way, the preceding theorem and its interior counterpart
from \cite{Wang} imply the following partial
regularity result.

\begin{cor}\label{cor:partial}
  Let $u\in W^{2,2}(B_4^+,N)$ be variationally biharmonic with respect
  to the Dirichlet datum $g\in C^\infty(B_4^+,N)$ on $T_4$. Then there is a
  subset $\Sigma\subset\Omega$ with $\HM^{m-4}(\Sigma)=0$ so that
  $u\in C^\infty(\Omega\setminus\Sigma,N)$.
\end{cor}

Finally, we have the following quantitative estimate. In the case of
harmonic maps, the corresponding result is due to Schoen
\cite{Schoen}. For the higher order case considered here, we employ a
technique that goes back to Moser \cite[Lemma 5.3]{Moser1}, see also
\cite[Lemma 5.3]{Scheven2}. 

\begin{lem}\label{lem:uniform}
  For every $\delta>0$, $\Lambda>0$, and $\alpha\in(0,1)$,
  there is a constant $\eps=\eps(\delta,\Lambda,\alpha,m,N)>0$ so that the
  following holds. Assume that $u\in C^4(\overline B_1^+,N)$ is a
  biharmonic map with Dirichlet values $g\in C^{4,\alpha}(B_1^+,N)$ on
  $T_1$ for which 
  \begin{equation}\label{smalldelta}
    \left\{\begin{array}{l}
    \displaystyle\sup_{B_\rho^+(a)\subset B_1^+}\,
    \rho^{2-m}\int_{B_\rho^+(a)}|Du|^2\,\d x<\eps\\[3.5ex]
    \quad\|g\|_{C^{4,\alpha}(B_1^+)}\le \Lambda
    \end{array}\right.
  \end{equation}
  holds true, then we have the estimate 
  \begin{equation*}
    \sum_{k=1}^4|D^ku(x)|^{1/k}\le \frac\delta{1-|x|}
   \mbox{\qquad for all $x\in B_1^+$.}
 \end{equation*}
\end{lem}
\begin{proof}
  We use the abbreviation $[u]_{C^4}(x):=\sum_{k=1}^4|D^ku(x)|^{1/k}$ for 
  $x\in B_1^+$. If the assertion of the lemma was not true, we could
  find sequences of biharmonic maps $u_i\in C^4(\overline B_1^+,N)$ and
  boundary values $g_i\in C^{4,\alpha}(B_1^+,N)$ with
  $(u_i,Du_i)=(g_i,Dg_i)$ on $T_1$ for all $i\in\N$ so that 
  \begin{align}\label{smalldelta-sequence}
    \left\{
    \begin{array}{l}  
     \displaystyle\sup_{B_\rho^+(a)\subset
      B_1^+}\,\rho^{2-m}\int_{B_\rho^+(a)}|Du_i|^2\,\d x\to0
      \qquad\mbox{as  }i\to\infty,\\[3.5ex]
      \quad\;\displaystyle\sup_{i\in\N}\|g_i\|_{C^{4,\alpha}(B_1^+)}\le \Lambda,
    \end{array}
    \right.
  \end{align}
  but for all $i\in\N$, we have
  \begin{align}\label{notsmall}
    \max_{0\le r\le 1}\,(1-r)\max_{\overline B_r^+}\,[u_i]_{C^4}>\delta.
  \end{align}
  
  For every $i\in\N$, we choose a radius $r_i\in[0,1)$ with
  \begin{equation*}
    (1-r_i)\max_{\overline B_{r_i}^+}\,[u_i]_{C^4}
    =
    \max_{0\le r\le 1}(1-r)\max_{\overline B_r^+}\,[u_i]_{C^4}
  \end{equation*}
  and then a point $x_i\in\overline B_{r_i}^+$ with
  \begin{equation*}
    [u_i]_{C^4}(x_i)=\max_{\overline B_{r_i}^+}\,[u_i]_{C^4}.
  \end{equation*}
  With these choices, we define scaling factors by 
  \begin{equation*}
    \lambda_i:=\frac\delta{2\,[u_i]_{C^4}(x_i)}
    \qquad\mbox{for }i\in\N.
  \end{equation*}
  We observe that \eqref{notsmall} implies $\lambda_i<\frac{1-r_i}2<1$.
  With these factors we define rescaled maps
  \begin{equation*}
    v_i(x):=u_i(x_i+\lambda_ix)
    \qquad\mbox{and}\qquad
    h_i(x):=g_i(x_i+\lambda_ix)
  \end{equation*}
  for $x\in\Omega_i$, where
  \begin{equation*}
    \Omega_i:=\left\{(x^{(1)},\ldots,x^{(m)})\in
    B_1\,:\,x^{(m)}\ge-\lambda_i^{-1}x_i^{(m)}\right\}\supset B_1^+. 
  \end{equation*}
  The rescaled maps satisfy 
  \begin{equation}\label{delta2}
   [v_i]_{C^4}(0)=\lambda_i[u_i]_{C^4}(x_i)= \frac\delta2
  \end{equation}
  by definition of $\lambda_i$. Moreover, because of
  $B_{\lambda_i}(x_i)\subset B_{(1+r_i)/2}$ and the choice of $r_i$,
  $x_i$, and $\lambda_i$, we infer 
  \begin{equation}\label{delta3}
    \max_{\overline \Omega_i}\,[v_i]_{C^4}
    \le \lambda_i\max_{\overline B_{(1+r_i)/2}^+}\,[u_i]_{C^4}
    \le
    \lambda_i\,\frac{1-r_i}{1-\frac{1+r_i}2}\,[u_i]_{C^4}(x_i)
    =
    2\lambda_i[u_i]_{C^4}(x_i)
     =\delta.
  \end{equation}
  From $\lambda_i<1$ and \eqref{smalldelta-sequence}$_2$, we obtain
  \begin{equation}
    \label{G1}
    \sup_{i\in\N}\|h_i\|_{C^{4,\alpha}(\Omega_i)}
    <
    \sup_{i\in\N}\|g_i\|_{C^{4,\alpha}(B_1^+)}
    \le \Lambda.
  \end{equation}
  By the scaling invariance of the biharmonic map equation, the maps
  $v_i$ are again biharmonic, which means
  by~\eqref{weakly-bi} that they satisfy a boundary value problem of
  the form  
  \begin{equation}\label{boundary-problem}
    \left\{
    \begin{array}{ll}
      \Delta^2v_i=\tilde f(v_i,Dv_i,D^2v_i,D^3v_i)=:f_i
      &\mbox{in }\Omega_i,\\[1.5ex]
      v_i=h_i,\ Dv_i=Dh_i&
      \mbox{on $B_1\cap\big\{x^{(m)}=-\lambda_i^{-1}x_i^{(m)}\big\}$}.
    \end{array}
    \right.
  \end{equation}
  Clearly, in the case that the last set is empty, the map $v_i$ satisfies the
  differential equation on the full ball $B_1$ and there is no
  boundary condition.
  From the form of the biharmonic map equation and \eqref{delta3}, we
  infer
  \begin{equation}
    \label{right-hand-side-bounded}
    \sup_{i\in\N}\|f_i\|_{C^1(\Omega_i)}<\infty.
  \end{equation}
  For a standard cut-off function $\zeta\in C^\infty_0(B_1,[0,1])$
  with $\zeta\equiv 1$ in $B_{1/2}$, we use classical Schauder
  estimates for the maps $\zeta v_i$ on the halfspaces
  $\R^m\cap\{x^{(m)}=-\lambda_i^{-1} x_i^{(m)}\}$. In this way, we deduce
  that~\eqref{boundary-problem} implies
  \begin{equation*}
    \|v_i\|_{C^{4,\alpha}(\Omega_i\cap B_{1/2})}
    \le
    c\big(\|h_i\|_{C^{4,\alpha}(\Omega_i)}+\|f_i\|_{C^{0,\alpha}(\Omega_i)}+\|v_i\|_{C^{3,\alpha}(\Omega_i)}\big).
  \end{equation*}
  In view of \eqref{delta3}, \eqref{G1}, and \eqref{right-hand-side-bounded}, we infer
  that the restrictions $v_i|_{B_{1/2}^+}$ are bounded in
  $C^{4,\alpha}(B_{1/2}^+,N)$, independently of $i\in\N$.
  Therefore, after passing to a
  subsequence, the Arzel\`a-Ascoli theorem
  yields the convergence
  \begin{equation*}
    v_i\to v\qquad\mbox{in $C^4(B_{1/2}^+,N)$, as $i\to\infty$},
  \end{equation*}
  for some limit map $v\in C^4(B_{1/2}^+,N)$. Now on the one hand, the
  identity~\eqref{delta2} implies
  \begin{equation}\label{notconstant}
    [v]_{C^4}(0)=\lim_{i\to\infty}[v_i]_{C^4}(0)=\frac\delta2,
  \end{equation}
  but on the other hand, the choice of $u_i$ according
  to~\eqref{smalldelta-sequence}$_1$ leads to the estimate
  \begin{align*}
    \int_{B_{1/2}^+}|Dv|^2\,\d x
    &=
    \lim_{i\to\infty}\int_{B_{1/2}^+}|Dv_i|^2\,\d x
    \le
    \lim_{i\to\infty}\int_{\Omega_i\cap B_{1/2}}|Dv_i|^2\,\d x\\
    &=
    \lim_{i\to\infty}\lambda_i^{2-m}\int_{B_{\lambda_i/2}^+(x_i)}|Du_i|^2\,\d
    x=0,
  \end{align*}
  which means that $v$ is constant on $B_{1/2}^+$. In view of
  \eqref{notconstant}, this yields the desired contradiction and
  completes the proof of the lemma. 
\end{proof}

By combining the preceding regularity results, we arrive at the following conclusion.

\begin{cor}
  \label{cor:uniform}
  We consider a ball $B_r^+(x_0)=B_r(x_0)\cap\R^m_+$ for some
  $x_0\in\R^m_+$ and $r\in(0,1]$. 
  There is a constant $\eps_0=\eps_0(\Lambda,\alpha,m,N)>0$, so that for every
  weakly biharmonic map $u\in W^{2,2}(B_r^+(x_0),N)$ that satisfies
  \eqref{stationary-bi} and attains the 
  Dirichlet datum $g\in C^\infty(B_r^+(x_0),N)$ on $T_r(x_0)$, the estimates  
  \begin{equation}\label{small}
    \left\{
    \begin{array}{l}
      \displaystyle r^{4-m}\int_{B_r^+(x_0)}\big(|\Delta u|^2
      +r^{-2}|Du|^2\big)\,\d x<\eps_0,\\[3.5ex]
      \displaystyle\|g\|_{C^{4,\alpha}(B_r^+(x_0))}\le \Lambda
    \end{array}
    \right.
  \end{equation}
  imply $u\in C^4(B_{r/2}^+(x_0),N)$ with $\|Du\|_{C^3(B_{r/2}^+(x_0))}\le
  c(\Lambda,\alpha,m,N)$. 
\end{cor}
\begin{proof}
  By a scaling argument, it suffices to consider the case $r=1$. 
  In the case that $B_{3/4}^+(x_0)$ intersects
  $\partial\R^m_+$,
  we apply Lemma~\ref{lemma:morrey} on balls $B_{1/4}^+(a)$ 
  for any $a\in B_{3/4}^+(x_0)\cap \partial\R^m_+$. In view of this
  lemma and assumption~\eqref{small}$_1$, we infer 
  \begin{align*}
    \sup_{B_\rho^+(y)\subset B_{R_0}^+(a)}\,
    \rho^{4-m}\int_{B_\rho^+(y)}\big(|D^2u|^2+\rho^{-2}|Du|^2\big)\,\d x
    \le
    cR_0^{2-m}\eps_0+cR_0,
  \end{align*}
  for any $R_0\in(0,\tfrac18)$.
  By choosing first $R_0$ and then $\eps_0$ small enough in dependence on
  $\Lambda,\alpha, m$, and $N$,
  we can ensure that the assumptions of
  Theorem~\ref{epsreg} and Lemma~\ref{lem:uniform} with $\delta=1$
  are satisfied on
  the ball $B_{R_0}^+(a)$,
  which imply that $u\in C^4(B_{R_0/2}^+(a))$ with 
  $\|Du\|_{C^3(B_{R_0/4}^+(a))}\le c$. 
  We note that the application of Lemma~\ref{lem:uniform} is possible
  after a suitable rescaling.
  Since the last estimate holds for
  any $a\in B_{3/4}^+(x_0)\cap \partial\R^m_+$, we deduce the asserted
  gradient estimates in every point $y\in B_{1/2}^+(x_0)$ with $y_m<\frac14
  R_0$. In the remaining case $y\in B_{1/2}^+(x_0)$ with $y_m\ge\tfrac14 R_0$, we apply the
  corresponding interior estimates on the ball $B_{R_0/4}(y)$,
  cf. \cite[Thm. 2.6]{Scheven1}. In this way, we arrive at the desired
  bound $\|Du\|_{C^3(B_{1/2}^+(x_0))}\le c(\Lambda,\alpha,m,N)$.
\end{proof}

\section{Compactness for sequences of variationally biharmonic maps}
\label{sec:compact}

\subsection{The defect measure}
Our goal is to prove compactness for
a sequence of variationally biharmonic maps $u_i\in
W^{2,2}(B_4^+,N)$ with respect to Dirichlet values $g_i\in
C^{4,\alpha}(B_4^+,N)$ on $T_4$. We assume that the sequence is
bounded in the sense that 
\begin{align}\label{W22-bound}
  \sup_{i\in\N}\,(\|D^2u_i\|_{L^2(B_4^+)}+\|Du_i\|_{L^2(B_4^+)})<\infty
  \ \quad\mbox{and}\ \quad
  \sup_{i\in\N}\|g_i\|_{C^{4,\alpha}(B_4^+)}<\infty.
\end{align}
More precisely, we consider the slightly more general case of maps
with \eqref{weakly-bi} and \eqref{stationary-bi} instead of
variationally biharmonic maps, since the properties \eqref{weakly-bi}
and \eqref{stationary-bi} are clearly preserved under strong
convergence in $W^{2,2}$.
In view of \eqref{W22-bound}, 
Lemma~\ref{lemma:morrey} implies the Morrey space bound
\begin{equation}\label{bound:morrey}
  \sup_{i\in\N}\Big(|D^2u_i\|^2_{L^{2,m-4}(B_2^+)}+\|Du_i\|^2_{L^{2,m-2}(B_2^+)}
  +\|Dg_i\|_{C^{4,\alpha}(B_2^+)}^2\Big)\le\Lambda
\end{equation}
for a constant $\Lambda\ge1$ that depends on $m,N$, and the sequences
$(u_i)$ and $(g_i)$. Here we used the Morrey type norms that are defined by 
\begin{equation}\label{def:morrey-norm}
  \|f\|_{L^{2,\lambda}(B_2^+)}^2
  :=
  \sup_{B_\rho^+(y)\subset
    B_2^+}\frac1{\rho^\lambda}\int_{B_\rho^+(y)}|f|^2\,\d x,
\end{equation}
for $f\in L^2(B_2^+)$. 
From now on, we restrict ourselves to the ball $B_2^+$ and assume that
a bound of the type~\eqref{bound:morrey} is valid. 
This bound implies in particular that the sequence of Radon
measures $\Lm\edge|\Delta u_i|^2$ is bounded on $\overline B_2^+$. 
Therefore, possibly after passing to a
subsequence, we infer a map $u\in W^{2,2}(B_2^+,N)$ and a Radon
measure $\nu$ on $\overline B_2^+$ with $u_i\wto u$ weakly in
$W^{2,2}(B_2^+,\R^L)$, strongly in $W^{1,2}(B_2^+,\R^L)$ and a.e., and moreover 
\begin{align*}
  \Lm\edge|\Delta u_i|^2\wsto\Lm\edge|\Delta u|^2+\nu
\end{align*}
weakly*  in the space of Radon measures, as $i\to\infty$. 
The lower semicontinuity of the $L^2$-norm 
with respect to weak convergence 
implies $\nu\ge 0$. We call the measure $\nu$ the 
\emph{defect measure} of the sequence $u_i$, since it detects a
possible lack of strong convergence in $W^{2,2}$, see Lemma
\ref{defect} below.
The pair $(u,\nu)$ can be considered 
as the limit
configuration of the sequence $u_i$. This motivates the following 
\begin{defn}\label{def:BLambda}
  For sequences of maps $u_i\in W^{2,2}(B_2^+,N)$ and nonnegative 
  Radon measures $\nu_i$ on $\overline B_2^+$, where $i\in\N_0$, 
  we write $(u_i,\nu_i)\Bto(u_0,\nu_0)$ as $i\to\infty$ if and only if 
  convergence holds in the following sense. 
  \begin{align*}\nonumber
    \left\{
    \begin{array}{cl}
   u_i\wto u_0&\mbox{weakly in $W^{2,2}(B_2^+,\R^L)$}\\
   u_i\to u_0&\mbox{strongly in $W^{1,2}(B_2^+,\R^L)$ and a.e.,}\\
   \Lm\edge|\Delta u_i|^2+\nu_i\wsto\Lm\edge|\Delta u_0|^2+\nu_0&
    \mbox{weakly$^*$ as Radon measures.}
    \end{array}\right.
 \end{align*}
 For the set of all limit configurations of biharmonic maps, we
 write
 \begin{equation*}
   \B:=\left\{
        (u,\nu)\left|
        \begin{array}{l}
           (u_i,0)\Bto (u,\nu)\mbox{ for maps $u_i\in
          W^{2,2}(B_2^+,N)$ that satisfy}\\
          \mbox{\eqref{weakly-bi} and \eqref{stationary-bi}, 
          attain the boundary values $g_i\in C^\infty(B_2^+,N)$}\\
          \mbox{in the sense 
          $(u_i,Du_i)=(g_i,Dg_i)$ on $T_2$, and that satisfy}\\
          \mbox{$\|D^2u_i\|_{L^{2,m-4}(B_2^+)}^2+\|Du_i\|_{L^{2,m-2}(B_2^+)}^2+\|g_i\|_{C^{4,\alpha}(B_2^+)}^2\le\Lambda$}
         \end{array}
       \right.\right\}
 \end{equation*}
  for a given constant $\Lambda\ge1$ and $\alpha\in(0,1)$.
  Here, $"0"$ denotes the zero measure and we used the Morrey norms
  defined in \eqref{def:morrey-norm}.
  For a given pair $\mu=(u,\nu)\in\B$, we define the {\em energy concentration
  set} $\Sigma_\mu$ as the set of points $a\in \overline B_2^+$ with the
  property
  \begin{equation*}
    \liminf_{\rho\searrow 0}\left(\rho^{4-m}\int_{B_\rho^+(a)}(|\Delta
      u|^2+\rho^{-2}|Du|^2)\,\d x+\rho^{4-m}\nu(B_\rho^+(a))\right)\ge \eps_0,
  \end{equation*}
  where the constant $\eps_0=\eps_0(\Lambda,\alpha,m,N)>0$
  is chosen according to Corollary \ref{cor:uniform}.
\end{defn}

The following lemma clarifies the meaning of the defect measure and the energy
concentration set.
\begin{lem}\label{defect}
  Assume that $u_i\in W^{2,2}(B_2^+,N)$, $i\in\N$, is a sequence of
  maps with \eqref{weakly-bi}, \eqref{stationary-bi}, boundary values
  \begin{equation*}
    (u_i,Du_i)=(g_i,Dg_i)\qquad\mbox{on $T_2$ in the sense of traces},
  \end{equation*}
  so that the bound
  \begin{equation}\label{Lambda-bound}
    \sup_{i\in\N}\big(\|D^2u_i\|_{L^{2,m-4}(B_2^+)}^2
  +\|Du_i\|_{L^{2,m-2}(B_2^+)}^2+\|g_i\|_{C^{4,\alpha}(B_2^+)}^2\big)\le\Lambda
\end{equation}
  is satisfied. Moreover, we assume that 
  $(u_i,0)\Bto(u,\nu)=:\mu$ as $i\to\infty$,
  for some $(u,\nu)\in\B$.
  Then there holds
  \begin{enumerate}
  \item  $u_i\to u$ in $C^3_{\textup{loc}}(
    B_2^+{\setminus}\Sigma_\mu,\R^L)$ as
         $i\to\infty$.
   \item   If the defect measure satisfies
    $\spt\nu\cap \overline B_1^+=\varnothing$, then 
    we have strong convergence $u_i\to u$ in
    $W^{2,2}(B_{1}^+,\R^L)$, as $i\to\infty$.
  \end{enumerate}
\end{lem}
\begin{proof}
  In order to prove (i), we choose an arbitrary point
  $a\in B_2^+{\setminus}\Sigma_\mu$.  By the definition of
  $\Sigma_\mu$, we may choose a $\rho\in(0,1)$ with
  \begin{equation*}
    \rho^{4-m}\int_{B_\rho^+(a)}(|\Delta u|^2
    +\rho^{-2}|Du|^2)\,\d x+\rho^{4-m}\nu(B_\rho^+(a))<\eps_0.
  \end{equation*}
  By slightly diminishing the value of $\rho$ if necessary, we can
  additionally achieve that $\nu(\partial B_\rho^+(a))=0$, because
  $\nu(\partial B_\rho^+(a))>0$ can hold at most for countably many
  values of $\rho\in(0,1)$. Using the convergence
  $(u_i,0)\Bto(u,\nu)$, we conclude
  \begin{align*}
    &\lim_{i\to\infty}\rho^{4-m}\int_{B_\rho^+(a)}(|\Delta u_i|^2
    +\rho^{-2}|Du_i|^2)\,\d x\\
    &\qquad=
    \rho^{4-m}\int_{B_\rho^+(a)}(|\Delta u|^2
    +\rho^{-2}|Du|^2)\,\d x+\rho^{4-m}\nu(B_\rho^+(a))
    <\eps_0.
  \end{align*}
  Therefore, Corollary \ref{cor:uniform} yields the uniform estimate 
  $\|u_i\|_{C^4(B_{\rho/2}^+(a))}\le c(\Lambda,\alpha,m,N)$ for all
  sufficiently large $i\in\N$, from which
  we infer by Arz\'ela-Ascoli's theorem that $u_i\to u$ holds in
  $C^3(B_{\rho/2}^+(a),N)$ as $i\to\infty$. Since
  $a\in B_2^+{\setminus}\Sigma_\mu$ was arbitrary,
  this implies (i).\\
  For the proof of (ii), we note that in the case of
  $\spt\nu\cap\overline B_1^+=\varnothing$,
  the set $\Sigma_\mu\cap \overline B_1^+$ is given by 
  \begin{equation*}
    \Sigma_\mu\cap \overline B_1^+=
    \Big\{y\in \overline B_1^+:\liminf_{\rho\searrow 0}
    \rho^{4-m}\int_{B_\rho^+(y)}\big(|\Delta u|^2+\rho^{-2}|Du|^2\big)\,\d x\ge\eps_0\Big\}.
  \end{equation*}
  Since the values of $u$ are contained in the bounded manifold $N$,
  we have $u\in L^\infty\cap W^{2,2}(B_1^+,N)$, which implies by
  Gagliardo-Nirenberg embedding that $u\in W^{1,4}(B_1^+,N)$. Therefore,
  we infer from 
  H\"older's inequality and \cite[Lemma 3.2.2]{Ziemer} that
  $\HM^{m-4}(\Sigma_\mu\cap \overline B_1^+)=0$. This implies that for any given $\eps>0$,
  we can choose a cover
  $A:=\cup_{k\in\N}B_{\rho_k}(a_k)\supset\Sigma_\mu\cap \overline B_{1}^+$
  of open balls with radii $\rho_k\in(0,1)$ and centers
  $a_k\in \overline B_{1}^+$ so that $\sum_{k\in\N}\rho_k^{m-4}<\eps$.
  Hence, the bound~\eqref{Lambda-bound}
  yields the following estimate for all $i\in\N$.
  \begin{eqnarray*}
    \int_{A\cap B_{1}^+}|D^2u_i-D^2u|^2\,\d x
    &\le& 2\sum_{k\in\N}\int_{B_{\rho_k}^+(a_k)}\big(|D^2u_i|^2+|D^2u|^2\big)\,\d x\\
    &\le& c\Lambda\sum_{k\in\N}\rho_k^{m-4}
          \le c\Lambda\eps.
  \end{eqnarray*}
  On the compact set
  $\overline B_{1}^+\setminus A\subset\overline
  B_{1}^+{\setminus}\Sigma_\mu$, the conclusion (i) implies $u_i\to u$
  in $C^3(\overline B_{1}^+\setminus A,\R^L)$. Hence,
  \begin{equation*}
    \limsup_{i\to\infty}
    \int_{B_{1}^+}|D^2u_i-D^2u|^2\,\d x
    \le c\Lambda\eps\,+\,\lim_{i\to\infty}\int_{B_{1}^+\setminus A}|D^2u_i-D^2u|^2\,\d x
    = c\Lambda\eps.
  \end{equation*}
  Since $\eps>0$ was arbitrary, we conclude that $u_i\to u$ holds in
  $W^{2,2}(B_{1}^+,\R^L)$ in the limit $i\to\infty$, as claimed.
\end{proof}

Next, we analyse the relation between the defect measure and the energy
concentration set. 
\begin{lem}\label{structure}
  There are positive constants $c_1,c_2$, depending only on $m$, 
  so that every pair $\mu=(u,\nu)\in\B$ satisfies
  \begin{equation}\label{equiv}
    c_1\eps_0\HM^{m-4}\edge(\Sigma_\mu\cap \overline B_1^+)
    \le\nu\edge\overline B_1^+ 
    \le c_2\Lambda\HM^{m-4}\edge(\Sigma_\mu\cap \overline B_1^+).
  \end{equation}
  Furthermore, $\Sigma_\mu$ is a closed set and $\Sigma_\mu=\sing(u)\cup\spt(\nu)$.
\end{lem}

\begin{proof}
  The inclusion $\sing(u)\cup\spt(\nu)\subset\Sigma_\mu$ holds by
  Corollary~\ref{cor:uniform} and 
  Lemma~\ref{defect}\,(i). For the converse inclusion, we assume that
  there is some point $a\in\Sigma_\mu\setminus\sing(u)$. By the choice of
  $a$, the functions $|\Delta u|$ and $|Du|$ are bounded on a neighborhood
  of $a$. Therefore, the definition of $\Sigma_\mu$ implies
  \begin{equation*}
    \liminf_{\rho\searrow 0}\rho^{4-m}\nu(B_\rho^+(a))\ge\eps_0,
  \end{equation*}
  from which we infer $a\in\spt(\nu)$. We have thus proven
  $\Sigma_\mu=\sing(u)\cup\spt(\nu)$, which implies in particular that
  $\Sigma_\mu$ is a closed set.\\
  Now we turn our attention to the proof of \eqref{equiv}. For a Borel
  set $A\subset\Sigma_\mu\cap\overline B_1^+$, we choose an arbitrary cover
  $\cup_{j\in\N}B_{\rho_j}(a_j)\supset A$ of balls with radii
  $\rho_j\in(0,1)$ and centers $a_j\in\overline B_1^+$. Since
  $(u,\nu)\in\B$, we infer that 
  \begin{equation*}
    \nu(A)\le\sum_{j\in\N}\nu(B_{\rho_j}(a_j))
    \le c_2\Lambda\sum_{j\in\N}\rho_j^{m-4}
  \end{equation*}
  holds true, for a constant $c_2=c_2(m)$. 
  Since the cover of $A$ was arbitrary, we conclude
  $\nu(A)\le c_2\Lambda\HM^{m-4}(A)$.
  For the proof of the first
  estimate in \eqref{equiv}, we choose a set $E\subset\Sigma_\mu$ with
  $\HM^{m-4}(E)=0$ and
  \begin{equation*}
    \lim_{\rho\searrow 0}\rho^{4-m}\int_{B_\rho^+(a)}\big(|\Delta u|^2+|Du|^4\big)\,\d x=0
    \qquad\mbox{for all }a\in \Sigma_\mu\setminus E.
  \end{equation*}
  A set with this property exists by \cite[Lemma 3.2.2]{Ziemer}. From
  the choice of $E$ and the definition of $\Sigma_\mu$, we know
  \begin{equation}\label{positive}
    \liminf_{\rho\searrow 0}\rho^{4-m}\nu(B_\rho^+(a))\ge\eps_0
    \qquad\mbox{for all }a\in\Sigma_\mu\setminus E.
  \end{equation}
  Let $\eps>0$ be given. By the definition of the Hausdorff measure,
  we may choose $\delta>0$ small enough to ensure
  \begin{equation*}
    \HM^{m-4}(A\setminus E)
    \le\eps+\inf\Big\{\alpha(m{-}4)\sum_{j\in\N}r_j^{m-4}\ \Big|\
    A\setminus E\subset\bigcup_{j\in\N}B_{r_j}(a_j),\ 0<r_j\le\delta\,\Big\},
  \end{equation*}
  where $\alpha(m-4)$ denotes the volume of the $(m-4)$-dimensional
  unit ball.  Since $\nu$ is a Radon measure, we may choose an open
  set $O_\eps\supset A$ with $\nu(O_\eps\cap\R^m_+)\le\nu(A)+\eps$.  We consider
  the family of balls $B_{\rho_j}(a_j)\subset O_\eps$ with
  $0<\rho_j\le\delta/5$, centers $a_j\in A\setminus E$ and the property
  \begin{equation}\label{positive2}
    \rho_j^{4-m}\nu(B_{\rho_j}^+(a_j))\ge\frac{\eps_0}2.
  \end{equation}
  By \eqref{positive}, the union of all balls with these properties
  covers the set $A\setminus E$.
  A Vitali type covering argument therefore yields
  the existence of a countable disjoint family
  $\{B_{\rho_j}(a_j)\}_{j\in\N}$ of balls with the property
  \eqref{positive2}, $B_{\rho_j}(a_j)\subset O_\eps$ and
  $A\setminus E\subset\,\cup_jB_{5\rho_j}(a_j)$, where
  $5\rho_j\le\delta$ for all $j\in\N$. This implies
  \begin{equation*}
    \HM^{m-4}(A)=\HM^{m-4}(A\setminus E)\le \eps+\alpha(m{-}4)\sum_{j\in\N}
    (5\rho_j)^{m-4}
  \end{equation*}
  by the choice of $\delta$. Since the balls $B_{\rho_j}(a_j)$ are pairwise disjoint and
  satisfy \eqref{positive2}, we can further estimate
  \begin{equation*}
    \sum_{j\in\N}\rho_j^{m-4}
    \le
    \frac2{\eps_0}\sum_{j\in\N}\nu(B_{\rho_j}^+(a_j))
    \le\frac2{\eps_0}\nu(O_\eps\cap\R^m_+)\le\frac2{\eps_0}\big(\nu(A)+\eps\big).
  \end{equation*}
  Combining the last two estimates, we arrive at
  $$\eps_0\HM^{m-4}(A)\le \eps_0\eps+c_1^{-1}(\nu(A)+\eps)$$
  with a constant $c_1=c_1(m)>0$. Since $\eps>0$ was arbitrary, we
  can omit the terms involving $\eps$ in the preceding estimate. 
  This completes the proof of the lemma.
\end{proof}

For the blow-up analysis of the defect measure, the case of a constant
limit map is of particular interest. In this case, the defect measure
inherits a monotonicity property from the sequence of biharmonic maps.
\begin{lem}\label{monodefect}
  We consider a pair $(c,\nu)\in\B$, where $c\in N$ denotes a constant
  map. More precisely, we assume that $(c,\nu)$ can be approximated by
  biharmonic maps $u_i$ in the sense $\B\ni (u_i,0)\Bto(c,\nu)$, where
  the approximating maps satisfy the boundary condition
  $(u_i,Du_i)=(g_i,Dg_i)$ on $T_2$, for 
  boundary values $g_i\in C^\infty(B_2^+,N)$ with 
  \begin{equation}\label{C4-conv-bdry-values}
    g_i\to c
    \mbox {\qquad in $C^4(B_2^+,N)$, as $i\to\infty$}. 
\end{equation}
  Then, the functions
  $$(0,1]\ni r\mapsto r^{4-m}\nu(B_r^+(a))$$
  are monotonically nondecreasing for every $a\in \overline
  T_1$.
  Moreover, there is a subsequence $\{i_j\}\subset\N$ so that
  \begin{equation}\label{Phi-converge}
    \left\{
    \begin{array}{ll}
      \displaystyle\lim_{j\to\infty}\Phi_{u_{i_j}}(a;r)=r^{4-m}\nu(B_r^+(a)),\\
      \displaystyle\lim_{j\to\infty} K_{i_j}\Psi_{u_{i_j}}(a;\rho,r)\to 0
    \end{array}
    \right.
  \end{equation}
  for a.e. $\rho,r\in(0,1]$, as $j\to\infty$, where
  the terms $\Phi_{u_i}$, $\Psi_{u_i}$ and $K_i$ are 
  defined in Theorem~\ref{thm:monotonicity}. 
\end{lem}
\begin{proof}
  From Theorem \ref{thm:monotonicity} we deduce that the approximating
  biharmonic maps $u_i$ satisfy the monotonicity
  formula~\eqref{bdry-monotonicity} with constants $\chi_i$, $K_i>0$
  that satisfy $\chi_i\to0$ and $K_i\to0$ as a consequence of
  $\|Dg_i\|_{C^3(B_2^+)}\to0$ in the limit
  $i\to\infty$. Discarding the
  non-negative integral on the left-hand side of the monotonicity
  formula, we deduce
  \begin{equation}\label{limit-Phi}
    \Phi_{u_i}(a;\rho)
    \le
    \Phi_{u_i}(a;r)+K_i\Psi_{u_i}(a;\rho,r)
  \end{equation}
  for all $i\in\N$ and a.e. radii $0<\rho<r<1$.  Next, for a.e. radius $r\in(0,1)$
  and $X(x):=x-a$, we may define
  \begin{equation*}
    f_i(r):=\int_{S_r^+(a)}\big(\partial_X|Du_i|^2+4|Du_i|^2-4r^{-2}|\partial_Xu_i|^2\big)d\Hm,
  \end{equation*}
  as well as
  \begin{equation*}
    g_i(r):= K_i r^{6-m}\int_{S_r^+(a)}|D^2u_i|^2\d \Hm.
  \end{equation*}
  Since $u_i\to c$ strongly in $W^{1,2}(B_2^+,\R^K)$, 
  $\sup_i\|D^2u_i\|_{L^2(B_2^+)}^2<\infty$, and $K_i\to0$, we have $f_i\to 0$ and
  $g_i\to0$ in $L^1([0,1])$, as $i\to\infty$. Therefore, there is a subsequence
  $\{i_j\}\subset\N$ with $f_{i_j}\to 0$ a.e. and $g_{i_j}\to0$ a.e.
  in $[0,1]$, as
  $j\to\infty$.  We conclude that for almost every $r\in[0,1]$,
  \begin{align*}
    \Phi_{u_{i_j}}(a;r)
    &=e^{\chi_{i_j}r}r^{4-m}\int_{B_r^+(a)}|\Delta u_{i_j}|^2\,\d x+
    e^{\chi_{i_j}r}r^{3-m}f_{i_j}(r)\\
    &\to r^{4-m}\nu(B_r^+(a))
  \end{align*}
  in the limit $j\to\infty$, if we choose in particular $r\in[0,1]$ in
  such a way that $\nu(\partial B_r(a))=0$.
  Similarly, we deduce 
  \begin{align*}
    &K_{i_j}\Psi_{u_{i_j}}(a;\rho,r)\\
    &\quad=K_{i_j}r+K_{i_j}\int_{B_r^+(a)\setminus
      B_\rho^+(a)}\bigg(\frac{|D^2u_{i_j}|^2}{|x-a|^{m-5}}+\frac{|Du_{i_j}|^2}{|x-a|^{m-3}}\bigg)\,\d x+g_{i_j}(\rho)+g_{i_j}(r)\\\nonumber
    &\quad\to0
  \end{align*}
  as $j\to\infty$, for a.e. $0<\rho<r<1$.
  The two preceding convergences imply assertion~\eqref{Phi-converge}. 
  Using this result in~\eqref{limit-Phi}, we infer
  \begin{equation*}
    \rho^{4-m}\nu(B_\rho^+(a))
    \le
    r^{4-m}\nu(B_r^+(a))
  \end{equation*}
  for a.e. $0<\rho<r\le 1$. Since the measures $\nu(B_\rho^+(a)),$ $\nu(B_r^+(a))$ depend
  left-continuously on $\rho,r\in[0,1]$, this implies the asserted
  monontonicity property.  
\end{proof}

\subsection{Existence of flat tangent measures}
We consider tangent pairs of $\mu=(u,\nu)\in\B$ in a boundary point
$a\in \overline T_1$ in the
following sense. For scaling factors $r\in(0,1)$ we define rescaled
versions $\mu_{a,r}=(u_{a,r},\nu_{a,r})$ by 
\begin{align}
  u_{a,r}(x)&:=u(a+rx)&&\mbox{\ \ \ \ for $x\in B_{1/r}^+$,}&
     \label{rescale}\\\nonumber
  \nu_{a,r}(A)&:=r^{4-m}\nu(a+rA)&&\mbox{\ \ \ \ 
   for every Borel set $A\subset B_{1/r}^+$}.&
\end{align}
We note that the definition of $\B$ implies
\begin{equation}\label{bound-rescalings}
  \int_{B_{2}^+}\big(|D^2u_{a,r}|^2+|Du_{a,r}|^2\big)\d x
  +
  \nu_{a,r}(\overline B_{2}^+)
  \le c(m)\Lambda
\end{equation}
for any $a\in\overline T_1$ and $r\in(0,\tfrac12)$. 
For a map $u_*\in W^{2,2}(B_2^+,N)$ and a Radon measure $\nu_*$ on $\overline B_2^+$, 
we call $\mu_*=(u_*,\nu_*)$ a \emph{tangent pair} of $\mu$  in $a\in\overline
T_1$ if there exists a sequence $r_i\searrow 0$ so that 
$\mu_{a,r_i}\Bto\mu_*$.
For the family of all tangent pairs of a given $\mu\in\B$ in a
boundary point, we write
\begin{equation*}
  \T(\mu):=\big\{\mu_*\,\big|\,\mu_{a,r_i}\Bto\mu_*\mbox{ for a sequence
               $r_i\searrow 0$  and some $a\in \overline T_1$}\big\}.
\end{equation*}
A standard diagonal sequence argument yields            
\begin{equation}\label{tangent}
 \T(\mu)\subset\B,
\end{equation}
cf. \cite[Lemma 3.3]{Scheven1}.
More precisely, assume that $\mu_{a,r_i}\Bto\mu_\ast$ as $i\to\infty$,
and let $u_k\in
W^{2,2}(B_2^+,N)$ be biharmonic maps as in Definition
\ref{def:BLambda} with $(u_k,0)\Bto \mu$ in the limit
$k\to\infty$. Then a diagonal sequence argument yields a sequence
$(k_i)_{i\in\N}$ in $\N$ with $(\tilde u_i,0):=((u_{k_i})_{a,r_i},0)\Bto \mu_\ast$ as $i\to\infty$,
which proves $\mu_\ast\in\B$, and thereby the claim~\eqref{tangent}. We note that the boundary
values $\tilde g_i:=(g_{k_i})_{a,r_i}\in C^\infty(B_2^+,N)$ of $\tilde
u_i$ satisfy
\begin{equation*}
  \|D\tilde g_i\|_{C^3}
  \le
  r_i\|Dg_{k_i}\|_{C^3}
  \le
  r_i\Lambda\to0
  \qquad\mbox{in the limit }i\to\infty.
\end{equation*}
This means that every tangent pair $\mu_\ast\in\T(\mu)$ is the limit of
biharmonic maps $\tilde u_i$ in the sense $(\tilde
u_i,0)\Bto\mu_\ast$, where the boundary values of $\tilde u_i$
converge to a constant $c_\ast\in N$ in the sense 
\begin{equation}
  \label{eq:3}
  \tilde g_i\to c_\ast
  \quad\mbox{in $C^4(B_2^+,N)$, as $i\to\infty$.}
\end{equation}

In the following lemma, we construct a flat defect measure by a double
blow-up procedure.
\begin{lem}\label{blow-up}
  Assume that $N$ does not carry a non-constant
  Paneitz-biharmonic $4$-sphere, and 
  that there is a pair $\mu=(u,\nu)\in\B$ with
  $\spt\nu\cap\overline B_1^+\neq\varnothing$. Then there
  exists a pair $(c_\ast,\bar\nu)\in\T(\mu)\subset\B$, where $c_\ast\in N$ represents a constant map
  and
  \begin{equation}
    \bar\nu=c_0\HM^{m-4}\edge\big(V\cap \overline B_2^+\big)\label{flat}
  \end{equation}
  for an $(m-4)$-dimensional subspace $V\subset\partial\R^m_+$ and a  constant $c_0>0$.  
\end{lem}

\begin{proof}
  Since $u\in W^{2,2}\cap W^{1,4}(\overline B_2^+,N)$, we know that for
  $\HM^{m-4}$-a.e. $a\in \overline T_1$,
  \begin{equation}\label{nodensity}
    \lim_{\rho\searrow 0}\,\rho^{4-m}\int_{B_\rho^+(a)}\big(|D^2u|^2+|Du|^4)\,\d x=0,
  \end{equation}
  cf. \cite[Lemma 3.2.2]{Ziemer}. Since $N$ does not carry a non-constant
  Paneitz-biharmonic $4$-sphere,
  we know from the analysis of the interior case in
  \cite[Thm. 1.6]{Scheven1} that
  $\spt\nu\subset\partial\R^n_+$. On the other hand, we have 
  $\spt\nu\cap\overline B_1^+\neq\varnothing$ by assumption, so that Lemma \ref{structure} implies
  $\HM^{m-4}(\Sigma_\mu\cap\overline T_1)>0$.
  Therefore, we may
  choose a point $a\in\Sigma_\mu\cap\overline T_1 $ with the property
  \eqref{nodensity}. As a consequence of \eqref{nodensity}, there is a 
  sequence $r_i\searrow 0$ so that the rescaled maps $u_{a,r_i}$ satisfy
  \begin{equation*}
    u_{a,r_i}\to c_\ast
    \qquad\mbox{strongly in $W^{2,2}(B_2^+,\R^L)$ and in $W^{1,4}(B_2^+,\R^L)$,}
  \end{equation*}
  as $i\to\infty$, for some constant $c_\ast\in N$. 
  In view of~\eqref{bound-rescalings}, by passing to a subsequence
  we can achieve the convergence
  $\mu_{a,r_i}\Bto (c_\ast,\nu_*)$ for some Radon measure $\nu_*$ on
  $\overline B_2^+$ as $i\to\infty$.
  The result of this first blow-up is a tangent pair
  $\mu_*:=(c_\ast,\nu_*)\in\T(\mu)\subset\B$ for which the terms
  $r^{4-m}\,\nu_*(B_r^+(a))$ depend monotonically nondecreasing
  on $r\in(0,1)$ for every $a\in\overline T_1$, see Lemma
  \ref{monodefect}. We point out that
  assumption~\eqref{C4-conv-bdry-values} of this lemma is satisfied 
  in view of~\eqref{eq:3}. 
  In particular,
  this monotonicity property implies that the $(m-4)$-dimensional density
  \begin{equation*}
    \Theta^{m-4}(\nu_*,a):=\lim_{r\searrow 0}\,r^{4-m}\nu_*(B_r^+(a))
  \end{equation*}
  exists for every $a\in\overline T_1$. Because of $(c_\ast,\nu_*)\in\B$
  we have $\Theta^{m-4}(\nu_*,a)\le \Lambda$ for every
  $a\in\overline T_1$.  
  As already noted above, the interior result \cite[Thm. 1.6]{Scheven1}
  implies 
  $\Sigma_{\mu_\ast}\subset\partial\R^m_+$.
  Since $\mu_*=(c_\ast,\nu_*)$ with a
  constant map $c_\ast$, we can characterize the set
  $\Sigma_{\mu_*}\cap\overline B_1^+$ by 
  \begin{equation*}
    \Sigma_{\mu_*}\cap\overline B_1^+
    =\big\{a\in \overline T_1\,:\,\Theta^{m-4}(\nu_*,a)\ge\eps_0\big\}.
  \end{equation*}
  Because Lemma \ref{structure} implies  $\nu_*(\overline
  B_1^+\setminus\Sigma_{\mu_*})=0$, we deduce 
  \begin{equation*}
    0<\eps_0\le \Theta^{m-4}(\nu_*,y)\le \Lambda\qquad\mbox{for
      $\nu_*$-a.e. $y\in\overline B_1^+$.}
  \end{equation*}
  The above property implies the existence of an
  $(m-4)$-flat tangent measure of $\nu_*$, see
  \cite[Thm. 14.18]{Mattila}. More precisely, for
  $\nu_*$-a.e. $y\in\overline B_1^+$, there is a sequence
  $\rho_j\searrow 0$ with $(\nu_*)_{y,\rho_j}\wsto\bar\nu$
  weakly$^\ast$ in the space of Radon measures,
  where $\bar\nu=c_0\HM^{m-4}\edge V$ for an
  $(m-4)$-dimensional subspace $V\subset\R^m$ and a constant $c_0>0$.
  As above, we infer from the interior case in
  \cite[Thm. 1.6]{Scheven1} that $\nu_*(\overline B_1^+\setminus
  \overline T_1)=0$, so that we can find a boundary point
  $y\in\overline T_1$ with the above property and with $V\subset\partial\R^n_+$. 
  Finally, we note that a diagonal sequence argument
  implies that tangent pairs
  of tangent pairs are again tangent pairs, which implies that 
  $(c_\ast,\bar\nu)\in\T(\mu)$. This completes the proof.
\end{proof}

\subsection{Construction of a biharmonic $4$-sphere or $4$-halfsphere}

\begin{thm}\label{sphere}
  Assume that there is a pair $\mu=(u,\nu)\in\B$ with 
  $\spt\nu\cap\overline B_1^+\neq\varnothing$. Then
  there exists a non-constant Paneitz-biharmonic sphere $v\in
  C^\infty(S^4,N)$ or a non-constant Paneitz-biharmonic half-sphere
  $C^\infty(S_+^4,N)$ with constant boundary values.
\end{thm}

\begin{proof}
  For the proof we will adapt techniques from \cite{Lin}, see also
  \cite{Scheven1} for the higher order case. Assume for
  contradiction that $N$ does not contain any non-constant Paneitz-biharmonic
  $4$-spheres.
  Then, after a suitable rotation, Lemma \ref{blow-up} yields the existence of a pair
  $\bar\mu=(c_\ast,\bar\nu)\in\T(\mu)\subset\B$, where $c_\ast\in N$ is a
  constant map and $\bar\nu$ is the measure on $\overline B_2^+$ 
  given by 
  \begin{equation*}
    \bar\nu
    =c_0\HM^{m-4}\edge\big(\overline B_2^{m-4}\x\{0\}\big)
  \end{equation*}
  for a positive constant $c_0$. Clearly, we have
  $\Sigma_{\bar\mu}=\overline B_2^{m-4}\times\{0\}$. By definition
  of $\B$ there are maps 
  $u_i\in W^{2,2}(B_2^+,N)$ and boundary values $g_i\in
  C^{\infty}(B_2^+,N)$ with $(u_i,Du_i)=(g_i,Dg_i)$ on
  $T_2$ in the sense of traces, so that 
  \begin{equation}
  \|D^2u_i\|_{L^{2,m-4}(B_2^+)}^2+\|Du_i\|_{L^{2,m-2}(B_2^+)}^2+\|g_i\|_{C^{4,\alpha}(B_2^+)}^2
  \le \Lambda\label{morrey-sequence}
\end{equation}
for any $i\in\N$, and the
  following convergence holds as $i\to\infty$.
  \begin{align}
    u_i\to&\ c_\ast&&\hspace{-8em}\mbox{in }W^{1,2}(B_2^+,N),&\label{W12}\\
    u_i\to&\ c_\ast&&\hspace{-8em}\mbox{in }
    C^3_{\mathrm{loc}}(B_2^+\setminus(B_2^{m-4}\times\{0\}),N),&\label{C2}\\
    g_i\to &\ c_\ast&&\hspace{-8em}\mbox{in }C^4(B_2^+,N),&\label{gC4}\\
    \Lm\edge|\Delta u|^2\wsto&\ c_0\HM^{m-4}\edge\big(\overline
    B_2^{m-4}\x\{0\}\big)&&
    \mbox{as measures on $\overline B_2^+$}.&\label{Mass}
  \end{align}
  In particular, the convergence \eqref{C2} follows from Lemma
  \ref{defect}\,(i), and the convergence \eqref{gC4} follows from~\eqref{eq:3}.
  
  In the sequel we will use the notation
  $\C_r(x)=\C_r(x',x''):=B_r^{m-4}(x')\x B^4_r(x'')$ for cylinders
  with barycenter $x=(x',x'')\in\R^{m-4}\x\R^4$, and $\C_r^+(x):=C_r(x)\cap\R^m_+$. The barycenter will
  be omitted in the notation if it is zero.
  For the proof of the theorem we will either construct a smooth biharmonic
  map $\R^4\to N$ or a biharmonic map $\R^4_+\to N$
  with constant boundary values 
  as the limit of a blow-up sequence
  \begin{equation}\label{rescaled}
    v_i(y):=u_i(p_i+\delta_i y)
    \mbox{ with $p_i=(p_i',p_i'')\in B^{m-4}_{1/4}\x B^4_{1/8}\cap\R^m_+$ and
      $\delta_i\searrow 0$,}
  \end{equation}
  where the parameters $p_i$ and $\delta_i$ will be carefully chosen below.  The
  above map is defined for $y\in\C_{R_i}$ with $y_m\ge
  -\delta_i^{-1}p_{i,m}$, where $R_i:=1/8\delta_i\to\infty$, and we
  write $p_{i,m}$ for the $m$-th component of the vector $p_i\in\R^m$.
  The construction is carried out in several steps.\\
 
  {\em Step 1. }We claim that after extracting a subsequence, there
  holds
  \begin{equation}\label{claim1}
    \sum_{k=1}^{m-4}\int_{B_{1/2}^+}|D\partial_k u_i|^2\,\d x\to 0,
    \qquad\mbox{as }i\to\infty.
  \end{equation}
  From the monotonicity formula stated in Theorem
  \ref{thm:monotonicity} we infer 
  \begin{align}\label{bdry-mono}
    &
    4\int_{B_r^+(a)\setminus B_\rho^+(a)}e^{\chi_i|x-a|}\frac{|D\partial_Xu_i|^2}{|x-a|^{m-2}}\,\d x\\\nonumber
    &\quad\le
    \Phi_{u_i}(a;r)-\Phi_{u_i}(a;\rho)
    +
    K_i\Psi_{u_i}(a;\rho,r)
  \end{align}
  for a.e. $0<\rho<r<1$ and every $a\in B_1^{m-4}\times\{0\}$, where we used the
  abbreviations $\Phi_{u_i}$ and $\Psi_{u_i}$ introduced in
  \eqref{Def-Phi} and \eqref{Def-Psi} and $X(x):=x-a$. Moreover, since
  $g_i\to c_\ast$ in $C^4$, Theorem \ref{thm:monotonicity} yields $K_i\to0$
  and $\chi_i\to0$ in the limit $i\to\infty$. 
  Lemma~\ref{monodefect} implies that after passing to
  a subsequence, 
  we have the convergence 
  \begin{align*}
    &\lim_{i\to\infty}\Big(\Phi_{u_{i}}(a;r)-\Phi_{u_{i}}(a;\rho)+K_i\Psi_{u_i}(a;\rho,r)\Big)\\
    &\qquad=
    r^{4-m}\bar\nu(B_r^+(a))-\rho^{4-m}\bar\nu(B_\rho^+(a))
    =0
  \end{align*}
  for a.e. $0<\rho<r<1$, where the last identity follows from the
  particular form of $\bar\nu$ and the fact $a\in B_1^{m-4}\times\{0\}$.
  Using the last formula to pass to the limit in \eqref{bdry-mono}, we deduce
  \begin{align*}
    &\limsup_{i\to\infty}r^{2-m}\int_{B_r^+(a)\setminus
      B_\rho^+(a)}|D\partial_Xu_i|^2\,\d x\\
    &\qquad\le
    \limsup_{i\to\infty}\int_{B_r^+(a)\setminus
      B_\rho^+(a)}e^{\chi_i|x-a|}\frac{|D\partial_Xu_i|^2}{|x-a|^{m-2}}\,\d x
    =0.
  \end{align*}
  Moreover, the Morrey bound \eqref{morrey-sequence} implies 
  \begin{align*}
  \limsup_{i\to\infty}r^{2-m}\int_{B_\rho^+(a)}|D\partial_Xu_i|^2\,\d x
  \le
  c\rho^{m-2}r^{2-m}\Lambda,
  \end{align*}
  which can be made arbitrarily small by choosing $\rho>0$ small
  enough. Combining the last two formulae, we arrive at 
  \begin{align*}
    &\lim_{i\to\infty}\int_{B_r^+(a)}|D\partial_Xu_i|^2\,\d x=0
  \end{align*}
  for a.e. $r\in(0,1]$, where $X(x)=x-a$. We apply this identity once
  with $a_0:=0$ and once with $a_k:=\tfrac14 e_k\in B_1^{m-4}\times\{0\}$ for
  $k\in\{1,\ldots,m-4\}$. In this way, we deduce
  \begin{align*}
    &\tfrac1{16}\int_{B_{1/2}^+}|D\partial_ku_i|^2\,\d x\\
    &\qquad\le
    2\int_{B_{1/2}^+}|D\langle Du_i,x\rangle|^2\,\d x
    +
    2\int_{B_{1/2}^+}|D\langle Du_i,x-\tfrac14 e_k\rangle|^2\,\d x
    \to0
  \end{align*}
  in the limit $i\to\infty$. This yields the claim~\eqref{claim1}.\\

  {\em Step 2 \textup{(}Choice of $p_i'$\textup{)}.} We claim that the parameters
  $p_i'\in B_{1/4}^{m-4}$ can be chosen with the properties
  \begin{align}
    \sup_{0<r\le1/8}\sum_{k=1}^{m-4}\,
    r^{4-m}\int_{B_r^{m-4}(p_i')\x B_{1/4}^4\cap\R^m_+}|D\partial_k
    u_i|^2\,\d x\to 0
    \qquad\mbox{as $i\to\infty$}\label{step2}
  \end{align}
  and
  \begin{align}
    \mbox{\ \ the maps $u_i$ are of class $C^\infty$ 
    in a neighborhood of $\{p_i'\}\x\overline{B_{1/4}^4}\cap\R^m_+$.}
    \label{step2-smooth}
  \end{align}
   For the proof of this claim, we consider the functions $f_i\in
   L^1(\R^{m-4})$ defined by 
  \begin{equation*} 
    f_i(x'):=\sum_{k=1}^{m-4}\int_{B_{1/4}^4\cap\R^4_+}|D\partial_k
    u_i(x',x'')|^2\,\d x''
    \qquad\mbox{for }x'\in B_{3/8}^{m-4}
  \end{equation*}
  and $f_i(x')=0$ otherwise. 
  Since $B_{3/8}^{m-4}\times B_{1/4}^4\subset
  B_{1/2}$, the convergence~\eqref{claim1} implies $f_i\to 0$ in
  $L^1(\R^{m-4})$, as $i\to\infty$. 
  The weak $L^1$-estimate for the Hardy-Littlewood maximal function
  implies 
  \begin{equation*}
    \mathcal{L}^{m-4}(\{x'\in\R^{m-4}:\mathcal{M}f_i(x')>\eps_i\})\le\frac{c(m)}{\eps_i}\|f_i\|_{L^1}
  \end{equation*}
   for every $\eps_i>0$, where we used the abbreviation
    \begin{equation*}
    \mathcal{M}f_i(x'):= \sup_{r>0}\
    r^{4-m}\int_{B_r^{m-4}(x')}f_i(y')\,\d y'\qquad\mbox{for }x'\in\R^{m-4}.
\end{equation*}
With the choices 
$\eps_i:=2c(m)\|f_i\|_{L^1}/\mathcal{L}^{m-4}(B_{1/4}^{m-4})\to0$, as
$i\to\infty$, the 
  above inequality implies
  \begin{equation}\label{max_func}
    \mathcal{L}^{m-4}(\{x'\in B_{1/4}^{m-4}:\mathcal{M}f_i(x')>\eps_i\})\le\frac12\mathcal{L}^{m-4}(B_{1/4}^{m-4}).
  \end{equation}
  Furthermore, the partial regularity result from Corollary
  \ref{cor:partial} implies 
  \begin{align}\nonumber
    &\mathcal{L}^{m-4}\left(\left\{x'\in B_{1/4}^{m-4}: 
            \mbox{there is a $x''\in\overline{B_{1/4}^4}\cap\R^4_+$
              with
            $(x',x'')\in\sing(u_i)$}\right\}\right) \\
    &\quad\le\mathcal{H}^{m-4}(\sing(u_i))=0.\label{reg}
  \end{align}
  Because of (\ref{max_func}) and (\ref{reg}), there are points $p_i'\in
  B_{1/4}^{m-4}$ satisfying \eqref{step2-smooth} and
  $\mathcal{M}f_i(p_i')\le\eps_i\to 0$ as $i\to\infty$. The latter property
  implies \eqref{step2}, so that the claim of Step 2 is verified.\\

  {\em Step 3 \textup{(}Choice of the scaling factors $\delta_i$\textup{)}.} The scaling
  factors have to be chosen carefully to make sure that the
  bi-energies of the rescaled maps neither tend to zero nor become
  unbounded, since we want to obtain a non-constant biharmonic map of
  finite bi-energy in
  the limit. In order to preserve a certain energy level during the blow-up, we consider the
  quantity
  \begin{equation*}
    \F_i(\delta):=
    \max_{z\in \overline{B^4_{1/4}}\cap\R^m_+}\,\delta^{4-m}\int_{\C_{\delta}^+(p_i',z)}\,
    \big(|\Delta u_i|^2+\delta^{-2}|Du_i|^2\big)\,\d x
  \end{equation*}
  for $\delta\in(0,1]$. 
  We may assume that $p_i'\to:p_0'\in\overline B_{1/4}^{m-4}$ as
  $i\to\infty$.  We will inductively construct a subsequence
  $i_j\in\N$, $j\in\N_0$, and a decreasing sequence
  $\delta_{i_j}\searrow 0$ with
  \begin{equation}
    \label{energylevel}
    \F_{i_j}(\delta_{i_j})=2^{1-m}\eps_0
    \qquad\mbox{for any $j\in\N$}.
  \end{equation}
  Note that we did not claim anything for the case $j=0$,
  so we can simply choose $i_0=1$
  and $\delta_0=\frac12$. 
  Now suppose that $i_{j-1}\in\N$ and $\delta_{i_{j-1}}>0$ have already been
  chosen for some $j\in\N$. For any
  $\delta_*\in(0,\frac12\delta_{i_{j-1}})$, we
  infer from \eqref{W12} and \eqref{Mass} that
  \begin{equation*}
    \lim_{i\to\infty}\,\delta_*^{4-m}\int_{\C_{\delta_*}^+(p_i',0)}
    \big(|\Delta u_i|^2+\delta_*^{-2}|Du_i|^2\big)\,\d x
    =\delta_*^{4-m}\bar\nu(\C_{\delta_*}^+(p_0',0))\ge\eps_0,
  \end{equation*}
  where the last estimate holds because of
  $(p_0',0)\in\Sigma_{\bar\mu}$ and
  $\C_{\delta_*}(p_0',0)\supset B_{\delta_*}(p_0',0)$.
  Choosing
  $i_j\in\N$ large enough, depending on $\delta_*$, we can thus
  achieve
  $$
  \F_{i_j}(\delta_*)\ge 2^{1-m}\eps_0.
  $$
  This fixes the index $i_j\in\N$. 
  On the other hand, we know from the choice of $p_i'$ in Step 2 that
  the map $u_{i_j}$ is smooth on
  $A_{\delta}:=\overline{B_{\delta}^{m-4}(p_{i_j}')\x
    B^4_{1/4+\delta}}\cap\R^m_+$ if $\delta>0$ is chosen small
  enough. Consequently,
  \begin{equation*}
    \F_{i_j}(\delta)\le \max_{A_\delta}\big(|\Delta u_{i_j}|^2+|Du_{i_j}|^2\big)\delta^{2-m}
    \Lm(\C_{\delta}^+)\le 2^{-m}\eps_0
  \end{equation*}
  if $\delta\in(0,\delta_*]$ is chosen small enough in dependence on
  $i_j$. Combining the last two estimates and applying the
  intermediate value theorem, we deduce the existence of a number
  $\delta_{i_j}\in(0,\frac12\delta_{i_{j-1}})$ with the property
  \eqref{energylevel}. This construction yields the desired sequence 
  $\delta_{i_j}\searrow0$ for $j\to\infty$ with \eqref{energylevel},
  which concludes Step 3.
  In what follows, we will denote the
  subsequence $\{i_j\}$ again by $\{i\}$ for simplicity.\\
  
  {\em Step 4 \textup{(}Choice of $p_i''$\textup{)}.} 
  We choose points
  $p_i''\in\overline{B_{1/4}^4}\cap\R^4_+$ for which $p_i:=(p_i',p_i'')$
  satisfies
  \begin{equation}\label{max0}
    \delta_i^{4-m}\int_{\C_{\delta_i}^+(p_i)}
    \big(|\Delta u_i|^2+\delta_i^{-2}|Du_i|^2\big)\,\d x
    =\F_i(\delta_i)=2^{1-m}\eps_0
  \end{equation}
  for all $i\in\N$, which is possible by~\eqref{energylevel} and the
  definition of $\F_i$.
  We claim that for all but finitely many values of
  $i\in\N$, we have $p_i''\in B_{1/8}^4$. Indeed, if this was not
  the case, after passing to a subsequence we would have 
  \begin{equation*}
    p_{i}\in\overline{B_{1/4}^{m-4}\x(B_{1/4}^4\setminus
      B_{1/8}^4)}\qquad\mbox{for any }i\in\N.
  \end{equation*}
  We consider a radius $R_0\le\frac1{20}$ 
  to be fixed later independently of $i\in\N$.
  We can assume 
  $\sqrt2\delta_i\le\frac{R_0}4$ by choosing $i\in\N$ large enough.
  We write $p_i^0$ for the orthogonal projection of $p_i$ onto
  $\partial\R^m_+$ and distinguish 
  between the cases $|p_i-p_i^0|<\frac{R_0}4$ and
  $|p_i-p_i^0|\ge\frac{R_0}4$. 
  In the first case, we observe that
  $\C^+_{\delta_{i}}(p_i)\subset B^+_{\sqrt 2\delta_i}(p_i)
  \subset B^+_{R_0/2}(p_i^0)$.
  Applying Lemma
  \ref{lemma:morrey} with the center $a=p_i^0\in T_1$, we thus infer
  \begin{align}\label{boundary-case}
    &\delta_{i}^{4-m}\int_{\C^+_{\delta_{i}}(p_i)}
    \big(|\Delta u_{i}|^2+\delta_{i}^{-2}|Du_{i}|^2\big)\,\d x\\\nonumber
    &\qquad\le
    cR_0^{4-m}\int_{B^+_{R_0}(p_i^0)}
    \big(|\Delta u_{i}|^2+R_0^{-2}|Du_{i}|^2\big)\,\d x
    +cR_0\\\nonumber
    &\qquad\le
    cR_0^{4-m}\int_{B^+_{5R_0/4}(p_i)}
    \big(|\Delta u_{i}|^2+R_0^{-2}|Du_{i}|^2\big)\,\d x+cR_0\\
    &\qquad\le
    cR_0^{2-m}\int_{B_{1/2}^{m-4}\x(B^4_{1/2}\setminus B^4_{1/16})\cap\R^m_+}
          \big(|\Delta u_i|^2+|Du_i|^2\big)\,\d x+cR_0,\nonumber
  \end{align}
  where we used the property $R_0\le\frac1{20}$ in the last step. In
  the remaining case $|p_i-p_i^0|\ge\frac{R_0}{4}$, we again assume
  $\sqrt2\delta_i\le\frac{R_0}4$, which implies
  $\C_{\delta_{i}}^+(p_i)\subset B_{\sqrt 2\delta_i}(p_i)
  \subset B_{R_0/4}(p_i)$. 
  Therefore, Lemma
  \ref{lemma:morrey} implies 
  \begin{align}
    \label{interior-case}
    &\delta_{i}^{4-m}\int_{\C^+_{\delta_{i}}(p_i)}
    \big(|\Delta u_{i}|^2+\delta_{i}^{-2}|Du_{i}|^2\big)\,\d x\\\nonumber
    &\qquad\le
    cR_0^{2-m}\int_{B_{R_0/2}(p_i)}
    \big(|\Delta u_{i}|^2+R_0^{-2}|Du_{i}|^2\big)\,\d x+cR_0\\\nonumber
    &\qquad\le
    cR_0^{2-m}\int_{B_{1/2}^{m-4}\x(B^4_{1/2}\setminus B^4_{1/16})\cap\R^m_+}
          \big(|\Delta u_i|^2+|Du_i|^2\big)\,\d x+cR_0.
  \end{align}
  Now, we recall the choice of
  $p_i''$, use either \eqref{boundary-case} or \eqref{interior-case}
  and then the convergences \eqref{W12} and \eqref{Mass}. In this way, we deduce
  \begin{eqnarray*}
    2^{1-m}\eps_0
    &=&\lim_{i\to\infty}\delta_{i}^{4-m}\int_{\C^+_{\delta_{i}}(p_i)}
        \big(|\Delta u_{i}|^2+\delta_{i}^{-2}|Du_{i}|^2\big)\,\d x\\
    &\le&cR_0^{2-m}\lim_{i\to\infty}\int_{B_{1/2}^{m-4}\x(B^4_{1/2}\setminus B^4_{1/16})\cap\R^m_+}
          \big(|\Delta u_i|^2+|Du_i|^2\big)\,\d x+cR_0\\
    &=& cR_0^{2-m}\,\bar\nu\big(B_{1/2}^{m-4}\x(B^4_{1/2}\setminus B^4_{1/16})\cap\R^m_+\big)+cR_0=cR_0.
  \end{eqnarray*}
  By choosing $R_0$ so small that $cR_0<2^{1-m}\eps_0$, we arrive at the desired
  contradiction. This yields the claim $p_i''\in B_{1/8}^4$.\\

  {\em Step 5 \textup{(}Blow-up\textup{)}.}
  As before, we denote the $m$-th component of the vector $p_i\in\R^m$
  by $p_{i,m}$. 
  We distinguish between the case $\delta_i^{-1}p_{i,m}\to\infty$
  and the case
  $\delta_i^{-1}p_{i,m}\to b$ for some $b\in[0,\infty)$
  in the limit $i\to\infty$. By passing to a
  subsequence, we can ensure that one of these two alternatives
  is satisfied.
  In the first case, we define rescaled maps $v_i$
  according to
  \begin{equation*}
    v_i(y):= u_i(p_i+\delta_iy),
    \qquad\mbox{for $y\in \C_{R_i}$ with $y_m\ge -\delta_i^{-1}p_{i,m}$,}
  \end{equation*}
  where $R_i:=1/8\delta_i\to\infty$. In the case
  $\delta_i^{-1}p_{i,m}\to\infty$, the domain of definition of
  $v_i$ contains arbitrarily large balls centered in the
  origin. Therefore, this case can be treated 
  analogously as in the interior situation, cf. the arguments
  following (3.24) in \cite{Scheven1}. In this way, it is possible to
  show that the limit map of the rescaled maps $v_i$ is of the form
  $\hat v(x',x'')=v(x'')$ for a non-constant biharmonic map
  $v\in C^\infty(\R^4,N)$ with $|D^2u|\in L^2(\R^4)$.
  Under stereographic projection, this map
  corresponds to a non-constant Paneitz-biharmonic map $v\in
  C^\infty(S^4,N)$, which is a contradiction to our assumptions. 

  We omit the details, which have been carried out in \cite{Scheven1},
  and present only the corresponding arguments in the boundary
  situation, i.e. in the case $\delta_i^{-1}p_{i,m}\to b<\infty$, as
  $i\to\infty$.
  In this situation, we
  rescale around the orthogonal projections $p_i^0\in\partial\R^m_+$
  of $p_i$, i.e. we define maps $v_i\in W^{2,2}(\C_{R_i}^+,N)$ by
  letting 
  \begin{equation*}
    v_i(y):= u_i(p_i^0+\delta_iy),
    \quad\mbox{for $y\in \C_{R_i}^+$ and $i\in\N$,}
  \end{equation*}
  where $R_i:=1/8\delta_i\to\infty$. For the corresponding rescaling
  of the boundary
  values, we write
  \begin{equation*}
    h_i(y):= g_i(p_i^0+\delta_iy),
    \quad\mbox{for $y\in \C_{R_i}^+$ and $i\in\N$.}
  \end{equation*}
We can assume in particular that
  $\delta_i^{-1}p_{i,m}<R_i$ for all $i\in\N$.  As a consequence of
  our construction, the maps $v_i$ have the following properties. By
  \eqref{step2} and the fact $(p_i^0)''\in B_{1/8}^4$, there holds
  \begin{equation}
    \label{con}
    \sum_{k=1}^{m-4}R^{4-m}\int_{B_R^{m-4}\x
      B_{R_i}^4\cap\R^m_+}|D\partial_kv_i|^2\,\d x\to 0
  \end{equation}
  as $i\to\infty$, for every $R>0$. Moreover, \eqref{max0} and the
  definition of $\F_i$ implies
  \begin{align}
    \label{max}
    &\int_{\C_1^+(0,\delta_i^{-1}p_{i,m})}\big(|\Delta v_i|^2+|Dv_i|^2\big)\,\d y\\
    &\qquad=\max_{z\in B^4_{R_i}\cap\R^4_+}\int_{\C_1^+(0,z)}\big(|\Delta v_i|^2+|Dv_i|^2\big)\,\d y
    =2^{1-m}\eps_0\nonumber
  \end{align}
  for all $i\in\N$. Finally, from \eqref{morrey-sequence} we infer
  \begin{equation}
    \label{mor}
    \sup_{B_\rho^+(a)\subset \C_{R_i}^+}\rho^{4-m}\int_{B_\rho^+(a)}
    \big(|D^2v_i|^2+\rho^{-2}|Dv_i|^2\big)\,\d y\le \Lambda
  \end{equation}
  for all $i\in\N$. From \eqref{mor} we deduce by
  Rellich's theorem, combined with a diagonal sequence argument, that
  there is a limit map $\hat v\in W^{1,2}_{\mathrm{loc}}(\R^m_+,N)$ with
  $v_i\to\hat v$ strongly in $W^{1,2}(B_R^+,N)$ for every $R>0$ and
  almost everywhere, after passing to a subsequence.  For a fixed
  $R>0$, we may additionally find a subsequence $\{i_j\}\subset\N$
  with $v_{i_j}\wto \hat v$ weakly in $W^{2,2}(B_R^+,\R^L)$. Therefore,
  \eqref{con} implies for every $1\le k\le m-4$
  \begin{equation}
    \label{conhat}
    \int_{B_R}|D\partial_k\hat
    v|^2\,\d y\le\lim_{j\to\infty}\int_{B_R}|D\partial_kv_{i_j}|^2\,\d y=0
    \qquad\mbox{for any }R>0.
  \end{equation}
  As a consequence, $\partial_k\hat v$ is constant, and since $N$ is
  compact, even $\partial_k\hat v\equiv 0$ for every $1\le k\le m-4$.
  This implies that there is a map
  $v\in W^{2,2}_{\textrm{loc}}(\R^4_+,N)$ with $\hat v(y',y'')=v(y'')$ for a.e.
  $(y',y'')\in\R^{m-4}\times\R^4_+$.
  Furthermore, we recall for further reference that the strong
  convergence $v_i\to\hat v$ in $W^{1,2}(B_R^+,N)$ implies
  \begin{equation}
    \label{con1}
    \lim_{i\to\infty}\int_{B_R^+}|\partial_k v_i|^2\,\d y=0
    \qquad\mbox{for all $R>0$ and $1\le k\le m-4$.}
  \end{equation}
  \ \\

  {\em Step 6 ($C^3$-convergence). }  In order to establish local
  $C^3$-convergence  $v_i\to \hat v$,
  we will show that the bi-energy of the maps
  $v_i$ is small on every cylinder $\C_{1/2}(a)$ for $a\in B_R^+$, if
  $i>i_0(R)$ is sufficiently large. More precisely, we will show that
  the quantities
  \begin{equation*}
    \CF(a):=\int_{\C_1^+}\big(|\Delta v_i|^2+|Dv_i|^2\big)(y+a)\psi(y)\,\d y
  \end{equation*}
  are small, where $a\in\C^+_{R_i-1}$ and
  $\psi\in C^\infty_{\mathrm{cpt}}(\C_1,[0,1])$ is a cut-off function with
  $\psi\equiv 1$ on $\C_{1/2}$ and $|D^2\psi|+|D\psi|\le c$ for some
  constant $c=c(m)$. For $1\le k\le m-4$, we will use the 
  test vector field $\xi:=\psi_ae_k$, where $\psi_a(y):=\psi(y-a)$, in
  the differential equation \eqref{stationary-bi} for the maps
  $v_i$. Note that the maps $v_i$ also satisfy \eqref{stationary-bi} by
  scaling invariance, and that the test
  vector field is admissible since $e_k$ is tangential to
  $\partial\R^m_+$. 
  Before applying the differential equation, we calculate
  \begin{eqnarray*}
    \frac{\partial}{\partial a_k}\,\CF(a)
    &=&-\int_{\R^m}\big(|\Delta v_i(y)|^2+|Dv_i(y)|^2\big)\partial_k\psi(y-a)\,\d y\\
    &=&-\int_{\R^m}\big(|\Delta v_i|^2\Div\xi-2Dv_i\cdot D\partial_kv_i\,\psi_a\big)\,\d y,
  \end{eqnarray*}
  using integration by parts in the last step. We re-write the first
  term in the last integral by an application of the
  differential equation \eqref{stationary-bi} for the maps $v_i$,
  with the result 
  \begin{align*}
    &\left|\frac{\partial}{\partial a_k}\,\CF(a)\right|^2\\
    &\qquad\le
     2\left|\int_{\R^m}
     4\Delta v_i\cdot \partial_\ell\partial_k v_i\,\partial_\ell\psi_a 
     +2\Delta v_i\cdot\partial_kv_i\Delta\psi_a
     -2(Dv_i\cdot D\partial_kv_i)\psi_a
     \,\d y\right|^2\\
   &\qquad\qquad+
   2\left|\int_{\R^m}
    2\Delta v_i\cdot \Delta \big[\Pi(v_i)(\psi_a\partial_kh_i)\big]     
     \,\d y\right|^2\\
   &\qquad\le c\|v_i\|^2_{W^{2,2}(\C_1^+(a))}\int_{\C_1^+(a)}
       \big(|D\partial_kv_i|^2+|\partial_kv_i|^2\big)\,\d y\\
   &\qquad\qquad
       +c\|\Delta
       v_i\|_{L^2(\C_1^+(a))}^2(1+\|v_i\|^2_{W^{2,2}(\C_1^+(a))})
       \|Dh_i\|^2_{C^2(\C_1^+(a))}.
  \end{align*}
  The right-hand side vanishes in the limit $i\to\infty$ because the sequence $v_i$
  is bounded in $W^{2,2}(\C_1^+(a))$ by \eqref{mor} and we have the convergences    
  \eqref{con}, \eqref{con1}, and $h_i\to c_\ast$ in $C^4(\C_1^+(a))$.
  Therefore, the above estimate implies 
  $\frac{\partial}{\partial a_k}\,\CF\to 0$ uniformly on $B_R^+$ for every $R>0$ and
  every $1\le k\le m-4$, as $i\to\infty$. Since we know furthermore
  $\CF((0,a''))\le2^{1-m}\eps_0$ for all $a''\in B_{R_i}^4\cap \R^4_+$ by
  \eqref{max}, we arrive at
  \begin{equation*}
    \int_{\C_{1/2}^+(a)}\big(|\Delta v_i|^2+|Dv_i|^2\big)\,\d y\le\CF(a)<2^{2-m}\eps_0
  \end{equation*}
  for all $a\in B_R^+$, if $i\ge i_0(R)$ is chosen sufficiently
  large. Applying Corollary~\ref{cor:uniform} on 
  $B_{1/2}^+(a)\subset\C_{1/2}^+(a)$, we infer the bound
  \begin{equation*}
    \sup_{i\in\N}\|Dv_i\|_{C^3(B^+_{1/4}(a))}\le c(\Lambda,\alpha,m,N)
    \qquad\mbox{for all }a\in B_R^+.
  \end{equation*}
  By Arz\'ela-Ascoli, this implies convergence $v_i\to\hat v$ in
  $C^3(B_R^+,N)$. Since we have already established almost everywhere
  convergence $v_i\to\hat v$, it is not necessary to pass to another
  subsequence. Therefore, the $C^3$-convergence holds on every ball
  $B_R^+$ with $R>0$.\\

  {\em Step 7 (Conclusion). }  Keeping in mind that
  $\delta_i^{-1}p_{i,m}\to b\in[0,\infty)$ as $i\to\infty$, we deduce from the
  $C^3$-convergence and the identity \eqref{max} that 
  \begin{align*}
    &\int_{\C_1^+(0,b)}\big(|\Delta\hat v|^2+|D\hat v|^2\big)\,\d y\\
    &\qquad
    =\lim_{i\to\infty}\int_{\C_1^+(0,\delta_i^{-1}p_{i,m})}\big(|\Delta v_i|^2+|Dv_i|^2\big)\,\d y=2^{1-m}\eps_0>0.
  \end{align*}
  Consequently, the map $\hat v$, and thus also the restriction $v=\hat v|_{\{0\}\x\R^4_+}$,
  is not constant. On the other hand, \eqref{mor} implies for any $R>0$
  \begin{equation*}
    \int_{B_R^4\cap\R^4_+}|D^2v|^2\,\d y
    =c(m)R^{4-m}\int_{\C_R^+}|D^2\hat v|^2\,\d y\le c(m)\Lambda,
  \end{equation*}
  which yields $D^2v\in L^2(\R^4_+,N)$.
  Since the maps $v_i$ are weakly biharmonic
  and converge to $\hat v$ in $C^3_{\mathrm{loc}}(\R^m_+,N)$, the map $\hat v$ is biharmonic as
  well, and so is its restriction $v=\hat v|_{\{0\}\x\R^4_+}$.
  As a consequence, we even have $v\in C^\infty(\R^4_+,N)$,
  cf.~\cite{LammWang}.
  Thus, $v\in C^\infty(\R^4_+,N)$ is a non-constant biharmonic map with
  finite bi-energy and 
  constant boundary values on $\partial\R^4_+$.
  By stereographic projection and the conformal invariance of the
  Paneitz-bi-energy, this corresponds to a smooth, non-constant
  Paneitz-biharmonic $4$-halfsphere with constant boundary values, cf.
  Lemma~\ref{lem:liouville-4}\,(ii).
  This completes
  the proof of the theorem. 
\end{proof}

\subsection{Proof of Theorem~\ref{thm:compact}}
  We consider a sequence $u_i\in W^{2,2}(B_4^+,N)$, $i\in\N$, of
  variationally biharmonic maps with respect to Dirichlet values
  $g_i\in C^\infty(B_4^+,N)$ on $T_4$ for which 
  \begin{equation*}
    \sup_{i\in\N}\big(\|u_i\|_{W^{2,2}(B_4^+)}+\|g_i\|_{C^{4,\alpha}(B_4^+)}\big)<\infty
  \end{equation*}
  holds true, for some $\alpha\in(0,1)$.
  In view of Lemma~\ref{lemma:morrey}, this estimate implies the uniform Morrey
  space bound 
  \begin{equation*}
    \sup_{i\in\N}\big(\|D^2u_i\|_{L^{2,m-4}(B_2^+)}+\|Du_i\|_{L^{2,m-2}(B_2^+)}\big)<\infty. 
  \end{equation*}
  This means that there exists
  a constant $\Lambda\ge1$ so that $(u_i,0)\in\B$ for all $i\in\N$. 
  Therefore, after passing to a subsequence we have convergence
  $(u_i,0)\Bto(u,\nu)\in\B$ in the sense of Definition
  \ref{def:BLambda}, in the limit $i\to\infty$. 
  Let us assume for
  contradiction that there is no strong subconvergence $u_i\to u$ in
  $W^{2,2}(B_{1}^+,\R^L)$. Then Lemma \ref{defect}\,(ii) implies
  $~\spt\nu\cap\overline B_1^+\neq\varnothing$.
  Therefore, Theorem \ref{sphere}
  yields the existence of a non-constant Paneitz-biharmonic $4$-sphere
  or non-constant Paneitz-biharmonic $4$-halfsphere
  with constant boundary values.
  But the existence of such maps is excluded by the assumptions of
  Theorem~\ref{thm:compact},
  so that the claimed strong convergence holds true.
  \hfill$\square$\\

  \section{Liouville type theorems for biharmonic maps on a half
    space}
  \label{sec:liouville}

  The next theorem excludes the existence of certain non-constant
  biharmonic maps which might occur as tangent maps in singular
  boundary points. We remark that since these maps are homogeneous of
  degree zero, they can also be interpreted as maps
  $v:S^{m-1}_+\to N$ that are biharmonic with respect to a certain
  Paneitz-bi-energy, cf. \cite[Lemma 5.1]{Scheven1}. We note that the
  case $m=5$, which corresponds to Paneitz-biharmonic $4$-halfspheres, 
  can not be treated with the same methods and will be postponed to
  Lemma~\ref{lem:liouville-4}. 
  
\begin{thm}\label{thm:Liouville}
  Let $m\ge 6$ and
  assume that $u\in W^{2,2}_{\mathrm{loc}}(\R^m_+,N)$ is a stationary
  biharmonic map that is homogeneous of degree zero and admits
  constant Dirichlet boundary values in the sense $(u,Du)=(c,0)$ on
  $\partial\R^m_+$.
  Moreover, we assume that $u$ is smooth in a neighbourhood of
  $\partial\R^m_+\setminus\{0\}$. Then $u$ is constant in $\R^m_+$. 
\end{thm}

\begin{proof}
  We consider the vector field $\xi\in C^\infty(\R^m_+,\R^m)$ defined
  by $\xi(x):=\eta(|x|)e_m$ for a function $\eta\in
  C^\infty_0((0,\infty),[0,\infty))$. Since this vector field does not vanish
  on $\partial\R^m_+$, it is not admissible in the differential
  equation \eqref{stationary-bi-interior}. However, because $u$ is
  smooth in a neighbourhood of $\partial\R^m_+\setminus\{0\}$, we
  infer from Gau\ss' theorem that
  \begin{align}\nonumber\label{PDE-stationary}
    \mathrm{I}:=&\int_{\R^m_+} \big(4\Delta
    u\cdot \partial_i\partial_ju\,\partial_i\xi_j
    +2\Delta u\cdot \partial_iu\,\Delta\xi_i
    -|\Delta u|^2\partial_i\xi_i\big)\,\d x\\\nonumber
    &\quad=
    -\int_{\partial\R^m_+}\big(4\Delta
    u\cdot \partial_m\partial_ju\,\xi_j
    +2\Delta u\cdot \partial_iu\,\partial_m\xi_i\\
    &\qquad\qquad\qquad\qquad
    -2\partial_m(\Delta u\cdot\partial_iu)\xi_i
    -|\Delta u|^2\xi_m\big)\,\d x=:\mathrm{II},
  \end{align}
  where we used the convention to sum all double indices from $1$ to
  $m$. Using the definition of $\xi$ and the fact $Du=0$ on
  $\partial\R^m_+$, we compute
  \begin{align}\nonumber\label{calc-II}
    \mathrm{II}
    &=
    -\int_{\partial\R^m_+}\eta(|x|)\big(4\Delta
    u\cdot \partial_m^2u
    -2\partial_m(\Delta u\cdot\partial_mu)
    -|\Delta u|^2\big)\,\d x\\\nonumber
    &=
    -\int_{\partial\R^m_+}\eta(|x|)\big(2\Delta
    u\cdot \partial_m^2u
    -|\Delta u|^2\big)\,\d x\\
    &=
    -\int_{\partial\R^m_+}\eta(|x|)\,|\partial_m^2 u|^2\,\d x.
  \end{align}
  In the last step, we used the fact that $u$ is constant on
  $\partial\R^m_+$, which implies $\Delta u=\partial^2_mu$ on
  $\partial\R^m_+$. Moreover, with
  the abbreviation $r=|x|$, the definition of $\xi$ implies
  \begin{align*}
    \mathrm{I}&=
    \int_{\R^m_+}\tfrac1r \eta'(r)4\Delta u\, x_i\partial_i\partial_mu\,\d x\\
    &\qquad+\int_{\R^m_+} 2\Delta u\cdot \partial_m u
          \,r^{1-m}\tfrac{\partial}{\partial r}(r^{m-1}\eta'(r))\,\d x\\
    &\qquad-\int_{\R^m_+} \tfrac1r \eta'(r) |\Delta u|^2x_m\,\d x\\
    &=: \mathrm{I}_1+\mathrm{I}_2+\mathrm{I}_3.
  \end{align*}
  Since $u$ is homogeneous of degree zero, the derivative
  $\partial_mu$ is homogeneous of degree $-1$, which implies
  $x_i\partial_i\partial_mu=-\partial_mu$. By the homogeneity of $u$,
  we thus infer
  \begin{align}\label{calc-I1}
    \mathrm{I}_1
    =
    -4\int_0^\infty r^{m-5}\eta'(r) \,\d r \int_{S_1^+}\Delta u\cdot \partial_mu \,d\mathcal{H}^{m-1}.
  \end{align}
  Using integration by parts, we re-write the term $\mathrm{I}_2$ to
  \begin{align}\label{calc-I2}
    \mathrm{I}_2
    &=
    2\int_0^\infty r^{-3}\tfrac{\partial}{\partial r}(r^{m-1}\eta'(r))\,\d r
    \int_{S_1^+} \Delta u\cdot \partial_m u\,d\mathcal{H}^{m-1}\\\nonumber
    &=
    6\int_0^\infty r^{m-5}\eta'(r)\,\d r
    \int_{S_1^+} \Delta u\cdot \partial_m u\,d\mathcal{H}^{m-1}.
  \end{align}
  At this point, we used the assumption that $\eta$ has compact
  support in $(0,\infty)$, which implies that the boundary terms
  vanish. Finally, we compute
  \begin{align}\label{calc-I3}
    \mathrm{I}_3
    =
    -\int_0^\infty r^{m-5}\eta'(r)\,\d r
    \int_{S_1^+} |\Delta u|^2x_m\,d\mathcal{H}^{m-1}.
  \end{align}
  Plugging \eqref{calc-II}, \eqref{calc-I1}, \eqref{calc-I2}, and
  \eqref{calc-I3} into \eqref{PDE-stationary}, we deduce
  \begin{align}\label{conclusion-1}
    &-\int_0^\infty r^{m-5}\eta'(r)\,\d r
    \int_{S_1^+} \big[|\Delta u|^2x_m-2\Delta
    u\cdot\partial_mu\big]\,d\mathcal{H}^{m-1}\\\nonumber
    &\qquad=
    -\int_{\partial\R^m_+}\eta(|x|)\,|\partial_m^2 u|^2\,\d x\le0.
  \end{align}
  Using the homogeneity of $u$ and Gau\ss' theorem, we furthermore
  compute
  \begin{align*}
    -\int_{S_1^+}2\Delta u\cdot\partial_mu\,d\mathcal{H}^{m-1}
    &=
    -(m-3)\int_{B_1^+}2\Delta u\cdot\partial_mu\,\d x\\
    &=
    (m-3)\int_{B_1^+}2 Du\cdot\partial_m Du\,\d x,
  \end{align*}
  where the boundary integrals vanish because of
  $\frac{\partial u}{\partial r}=0$ on $S_1^+$ and $\partial_mu=0$ on
  $\partial\R^m_+$. Another application of Gau\ss' theorem implies
  \begin{align*}
    -\int_{S_1^+}2\Delta u\cdot\partial_mu\,d\mathcal{H}^{m-1}
    &=
    (m-3)\int_{B_1^+}\partial_m|Du|^2\,\d x\\
    &=
    (m-3)\int_{S_1^+}x_m|Du|^2 \,d\mathcal{H}^{m-1}.
  \end{align*}
  We use this identity in \eqref{conclusion-1} and integrate by parts
  with respect to $r$. In this way, we arrive at
    \begin{align}\label{final-liouville}
    &(m-5)\int_0^\infty r^{m-6}\eta(r)\,\d r
    \int_{S_1^+} x_m\big[|\Delta u|^2+(m-3)|Du|^2\big]
    \,d\mathcal{H}^{m-1}\\\nonumber
    &\qquad=
    -\int_{\partial\R^m_+}\eta(|x|)\,|\partial_m^2 u|^2\,\d x\le0.
  \end{align}
  Since $m>5$ and $\eta\ge0$, we infer $|Du|=0$ on $S_1^+$. By the
  homogeneity of $u$, this implies that $u$ is constant on $\R^m_+$, as claimed. 
\end{proof}

\begin{rem}\label{rem:m=5}
  \textup{For further reference we note that in the case $m=5$,
  the preceding proof yields that $D^2u\equiv0$ holds true on $\partial\R^5_+$
  for any map $u\in W^{2,2}_{\mathrm{loc}}(\R^5_+,N)$ as in
  Theorem~\ref{thm:Liouville}.
  For the verification of this assertion, we observe that
  the condition $m>5$ was only used in the very last
  step of the proof, so that identity \eqref{final-liouville} also holds in the
  case $m=5$. This identity implies $\partial_5^2u\equiv0$ on $\partial\R^5_+$,
  and since $Du\equiv0$ holds on $\partial\R^5_+$ by assumption, we deduce
  that all second derivatives vanish on the boundary. }
\end{rem}

In the case $m=5$, we obtain the result corresponding to
Theorem~\ref{thm:Liouville} under the additional assumption
of non-existence of Paneitz-biharmonic $4$-halfspheres.

\begin{lem}\label{lem:liouville-4}
  Assume that the manifold $N$ does not carry any non-constant
  Paneitz-biharmonic $4$-halfspheres with constant boundary
  values. Then the following two statements holds true.
  \begin{enumerate}
  \item Any biharmonic map $u\in C^\infty(\R^5_+\setminus\{0\},N)$
    that is homogeneous of degree zero and attains constant
    boundary values $(u,Du)=(c,0)$ on $\partial\R^5_+$ is constant.
  \item Any biharmonic map $w\in C^\infty(\R^4_+,N)$ with finite
    bi-energy and constant
    boundary values $(w,Dw)=(c,0)$ on $\partial\R^4_+$ is constant.
  \end{enumerate}
\end{lem}
\begin{proof}
 (i) Since $u$ is homogeneous of degree zero, we have 
 \begin{equation}\label{Bi-Laplace}
   \Delta^2u=\big(\tfrac{\partial^2}{\partial
     r^2}+\tfrac{4}{r}\tfrac{\partial}{\partial
     r}+\tfrac1{r^2}\Delta_S)\tfrac1{r^2}\Delta_Su
   =\tfrac1{r^4}(\Delta_S^2u-2\Delta_Su)
 \end{equation}
 in $\R^5_+\setminus\{0\}$, where $\Delta_S$ denotes the
 Laplace-Beltrami operator on $S^4_+$, and we abbreviated $r=|x|$.
 Since $u$ is biharmonic on $\R^5_+\setminus\{0\}$, we deduce that
 the restriction $v:=u|_{S^4_+}$ is Paneitz-biharmonic on $S^4_+$. Hence,
 the map $v$ is constant by assumption, and so is its homogeneous
 extension $u$. 
  
 (ii)  By stereographic projection, the biharmonic map $w\in
  C^\infty(\R^4,N)$ gives rise to a map $v\in
  C^\infty(S^4_+\setminus\{e_1\},N)$. The conformal invariance of the
  Paneitz-bi-energy, see \cite{Paneitz,Chang}, implies that 
  \begin{equation*}
    P_{S^4_+}(v)=\int_{\R^4_+}|\Delta w|^2\,\d x<\infty
  \end{equation*}
  and that $v\in W^{2,2}(S^4_+,N)$ is a critical point of the
  Paneitz-bi-energy. More precisely, the map $v$ is Paneitz-biharmonic
  on $S^4_+\setminus\{e_1\}$. To conclude, it remains to show that the
  singularity in $e_1$ can be removed. To this end, we consider the homogeneous
  extension $u(x):=v(\frac{x}{|x|})$. The identity~\eqref{Bi-Laplace}
  implies that $u\in W^{2,2}_{\mathrm{loc}}(\R^5_+,N)$ is biharmonic
  on $\R^5_+\setminus(\R_{\ge0}e_1)$. Since $\R_{\ge0}e_1$ is a
  one-dimensional set and is therefore negligible with respect to the
  $W^{2,2}$-capacity, we infer that $u$ is weakly biharmonic on
  $\R^5_+$. Moreover, since $u$ is the homogeneous extension of a map
  in $W^{2,2}\cap W^{1,4}(S^4_+,N)$, for any $\eps>0$ we can choose a radius $\delta>0$
  sufficiently small to ensure
  \begin{equation}\label{small-Morrey}
    \sup_{B_\rho^+(y)\subset
      B_\delta(e_1)}\frac1\rho\int_{B_\rho^+(y)}\big(|D^2u|^2+|Du|^4\big)\d
    x\le\eps.
  \end{equation}
  Therefore, the $\eps$-regularity result from \cite[Lemma 3.1]{GongLammWang}
  can be applied, which yields that $v$ is smooth in a neighbourhood of
  $e_1$. We point out that in view of \eqref{small-Morrey}, no
  monotonicity formula is required for this result. Consequently, we
  have shown that $v\in C^\infty(S^4_+,N)$ is a Paneitz-biharmonic map
  with constant Dirichlet boundary values. By assumption, the map $v$
  is constant, and so is the map $w\in
  C^\infty(\R^4_+,N)$. This completes the proof of the lemma. 
\end{proof}

We end this section with an example of a target manifold for which our
standing assumptions on the non-existence of Paneitz-biharmonic spheres are satisfied.

\begin{prop}\label{prop:flat-torus}
  Let $N=S_{r_1}^1\times\ldots\times S_{r_n}^1\subset\R^{2n}$ be a
  flat torus, where $n\ge1$ and $r_i>0$ for $i=1,\ldots,n$. Then $N$
  does neither carry a non-constant Paneitz-biharmonic $4$-sphere nor
  a non-constant Paneitz-biharmonic $4$-halfsphere with constant
  boundary values. 
\end{prop}

\begin{proof}
  The assertion concerning the full spheres has already been proved in
  \cite[Lemma 5.4]{Scheven1}. Here, we demonstrate that the same argument can
  be applied to yield the constancy of the corresponding halfspheres
  with constant boundary values.

  To this end, we consider a Paneitz-biharmonic map $u\in C^\infty(S^4_+,N)$ with
  constant boundary values in the sense $(u,Du)=(c,0)$ on $\partial
  S^4_+$. We extend $u$ to a function on $\R^5_+\setminus\{0\}$
  that is homogeneous of
  degree zero. Because of identity~\eqref{Bi-Laplace}, the extended
  function, which we still denote by $u$,
  is biharmonic on $\R^5_+\setminus\{0\}$. Therefore, Remark~\ref{rem:m=5}
  implies $D^2u\equiv0$ on $\partial\R^5_+$.   
  For every vector field $V\in C^\infty(S^4_+,\R^{2n})$
  with $V(x)\in T_{u(x)}N$ for every $x\in S^4_+$, the fact
  $\Delta_S^2u(x)-2\Delta_S u(x)\perp T_{u(x)}N$ for all $x\in S^4_+$
  and two integrations by parts imply
  \begin{align}\label{weak-Euler}
    0&=\int_{S^4_+}(\Delta_S^2u-2\Delta_Su)\cdot V\,\d\HM^4\\\nonumber
    &=\int_{S^4_+}(\Delta_S u\cdot \Delta_S V-2\Delta_Su\cdot V)\,\d\HM^4
    +\int_{\partial S^4_+}\big(
    \Delta_Su\cdot\partial_5V-\partial_5\Delta_S u\cdot V\big)\d\HM^3.
    \end{align}
    Since $\Delta_Su=0$ on $\partial S^4_+$, the boundary integral
    vanishes for every test vector field 
    with $V\equiv0$ on
    $\partial S^4_+$. In particular, this holds true for the choice
    $V:=\Delta_S^\top u$, where $\Delta_S^\top$ denotes the tangential
    part of the Laplace-Beltrami operator on $S^4_+$. We assume
    that $V$ is extended to a function on $\R^5_+\setminus\{0\}$
    that is homogeneous
    of degree zero. More precisely, this means that
    $V(x)=|x|^2\nabla^u_{e_i}\partial_iu(x)$ for any
    $x\in\R^5_+\setminus\{0\}$, where we used the notation $\nabla^u$ for
    the covariant derivative on the 
    bundle $u^\ast\mathrm{T}N$.  Since $u$ and $V$ are both homogeneous of degree
    zero, we have
    \begin{equation*}
      \Delta_S u
      =
      \partial_i\partial_i u
      =
      \Delta_S^\top u+A(u)(\partial_iu,\partial_iu)
    \end{equation*}
    and
    \begin{equation*}
      \Delta_SV
      =
      \partial_k\partial_kV
      =
      \partial_k\big[\nabla^u_{e_k}V+A(u)(\partial_ku,V)\big],
    \end{equation*}
    where $A$ denotes the second fundamental form of the submanifold
    $N\subset\R^{2n}$.
    Moreover, we employ the usual summation convention and sum all
    double indices from $1$ to $5$.
    Consequently, choosing $V:=\Delta_S^\top u$ in~\eqref{weak-Euler},
    we infer 
    \begin{align}\label{euler-lagrange}
      0&=
      \int_{S^4_+}(\Delta_S u\cdot\Delta_S V-2\Delta_Su\cdot
      V)\,\d\HM^4\\
      \nonumber
      &=\int_{S^4_+}
      \Big(\Delta_S^\top u\cdot \partial_k \nabla^u_{e_k}V
      +
      \Delta_S^\top u\cdot \partial_k[A(u)(\partial_ku,V)]\Big)
      \d\HM^4\\\nonumber
      &\qquad+
      \int_{S^4_+}\Big(A(u)(\partial_iu,\partial_iu)\cdot \partial_k\nabla^u_{e_k}V
      +
      A(u)(\partial_iu,\partial_iu)\cdot\partial_k[A(u)(\partial_ku,V)\big]\Big)\,\d\HM^4\\\nonumber
      &\qquad-\int_{S^4_+}2\Delta_S^\top u\cdot V\d\HM^4.
  \end{align}
  In the first and the third term appearing on the right-hand side, we use the
  identity 
  \begin{equation*}
     \partial_k\nabla^u_{e_k}V
     =
     \nabla^u_{e_k}\nabla^u_{e_k}V+A(u)(\partial_ku,\nabla^u_{e_k}V),
   \end{equation*}
   which follows from the definition of the covariant derivative and
   the second fundamental form.
  Next, using the equations by Weingarten and Gau\ss\ and keeping in mind
  that the Riemannian curvature of $N$ vanishes, we deduce
   \begin{align*}
     \Delta_S^\top u\cdot \partial_k[A(u)(\partial_ku,V)]
     &=
     -A(u)(\partial_ku,\Delta_S^\top u)\cdot A(u)(\partial_ku,V)\\
     &=
     -A(u)(\partial_ku,\partial_ku)\cdot A(u)(\Delta_S^\top u,V).
   \end{align*}
   Finally, because the flat torus has a parallel second fundamental
   form, i.e. $\nabla^\perp A\equiv 0$, we can compute
   \begin{align*}
     \partial_k\big[A(u)(\partial_ku,V)\big]
     &=
     A(u)(\Delta_S^\top u,V)+A(u)(\partial_ku,\nabla^u_{e_k}V).
   \end{align*}
   Using the preceding observations in~\eqref{euler-lagrange}, we deduce
   \begin{align}\label{flat-torus-1}
     0&=
     \int_{S^4_+}\big[\Delta_S^\top u\cdot
     \nabla^u_{e_k}\nabla^u_{e_k}V
     -2\Delta_S^\top u\cdot V\big]\d\HM^4\\\nonumber
     &\quad+
     2\int_{S^4_+}A(u)(\partial_iu,\partial_iu)\cdot
     A(u)(\partial_ku,\nabla^u_{e_k}V)\,\d\HM^4\\\nonumber
     &=:\mathrm{I}+\mathrm{II}. 
   \end{align}
   Next, we use the identity
   $\nabla^u_{e_k}\nabla^u_{e_k}V=\trace_S(\nabla^u\nabla^uV)$ on
   $S^4_+$,
   where we abbreviated $\trace_S$ for the trace on 
   $\mathrm{T}S^4_+$.
   Integrating by parts and recalling the definition of $V$, we then compute
   \begin{equation}
     \label{eq:I}
     \mathrm{I}
     =
     -\int_{S^4_+}\big[|\nabla^u\Delta_S^\top u|^2+2|\Delta_S^\top
     u|^2\big]\d\HM^4\le
     -2\int_{S^4_+}|\Delta_S^\top
     u|^2\d\HM^4.
   \end{equation}
   Since $N$ has vanishing Riemannian curvature, we have the following
   identity on $S^4_+$.
   \begin{align*}
     \nabla_{e_k}^uV
     &=
     \nabla_{e_k}\big(|x|^2\nabla^u_{e_\ell}\partial_\ell u\big)\\
     &=
     \nabla_{e_\ell}^u\nabla^u_{e_\ell}\partial_k u
     +2x_k\nabla_{e_\ell}^u\partial_\ell u\\
     &=
     \trace_S(\nabla^u\nabla^u\partial_ku)
     +\big(\tfrac{\nabla^2}{\partial r^2}+4\tfrac{\nabla}{\partial
       r}\big)\partial_ku+2x_k\nabla_{e_\ell}^u\partial_\ell u\\
     &=
     \trace_S(\nabla^u\nabla^u\partial_ku)
     -2\partial_ku
     +2x_k\nabla_{e_\ell}^u\partial_\ell u,
   \end{align*}
   where the last step follows from the fact
   that $\partial_ku$ is homogeneous of degree $-1$. 
   We use this identity, recall the fact $x_k\partial_ku=0$
   and integrate by parts. Since the second fundamental form is parallel, we infer
   \begin{align}\label{flat-torus-2}
     \mathrm{II}&=
     2\int_{S^4_+}A(u)(\partial_iu,\partial_iu)\cdot
     \trace_S \big[A(u)(\partial_ku,\nabla^u\nabla^u\partial_ku)\big]\,\d\HM^4\\\nonumber
     &\qquad-4\int_{S^4_+}|A(u)(\partial_iu,\partial_iu)|^2\,\d\HM^4\\\nonumber
     &=
     -
     2\int_{S^4_+}A(u)(\partial_iu,\partial_iu)\cdot
     \trace_S\big[
     A(u)(\nabla^u\partial_ku,\nabla^u\partial_ku)\big]\,\d\HM^4\\\nonumber
     &\qquad-4\int_{S^4_+}\trace_S\big[A(u)(\partial_iu,\nabla^u\partial_iu)\cdot
     A(u)(\partial_ku,\nabla^u\partial_ku)\big]\,\d\HM^4\\\nonumber
     &\qquad
     -4\int_{S^4_+}|A(u)(\partial_iu,\partial_iu)|^2\,\d\HM^4.
   \end{align}
   We note that the boundary term appearing in the integration by parts vanishes
   because of $Du\equiv0$ on $\partial\R^5_+$. The second
   integral on the right-hand side is non-positive, and so is the
   first one, because the
   Gau\ss\ equations on the flat manifold $N$ imply 
   \begin{align*}
     &-2A(u)(\partial_iu,\partial_iu)\cdot
     \trace_S\big[A(u)(\nabla^u\partial_ku,\nabla^u\partial_ku)\big]\\
     &\qquad=
     -2\trace_S\big[A(u)(\partial_iu,\nabla^u\partial_ku)\cdot
     A(u)(\partial_iu,\nabla^u\partial_ku)\big]
     \le0.
   \end{align*}
   Consequently, equation~\eqref{flat-torus-2} implies
   \begin{equation}
     \label{eq:II}
     \mathrm{II}
     \le
     -4\int_{S^4_+}|A(u)(\partial_iu,\partial_iu)|^2\,\d\HM^4.
   \end{equation}
   Using~\eqref{eq:I} and~\eqref{eq:II} in~\eqref{flat-torus-1}, we arrive at
   \begin{equation*}
     \int_{S^4_+}\big[2|\Delta_S^\top u|^2
     +
     4|A(u)(\partial_iu,\partial_iu)|^2\big]\,\d\HM^4
     \le0.
   \end{equation*}
   This implies that both the tangential and the normal part of
   $\Delta u$ vanish on $S^4_+$, which means that 
   $u:\R^5_+\setminus\{0\}\to\R^{2n}$ is a harmonic function.
   Classical regularity results for harmonic functions now yield
   $u\in C^\infty(\R^5_+,N)$. Since $u$ is homogeneous of degree
   zero, this implies that it is a constant map, as claimed. This
   completes the proof of the proposition. 
\end{proof}

\section{Full boundary regularity}
\label{sec:dim-red}

\subsection{Properties of tangent maps}
The strategy for the proof of the full boundary regularity is to
apply the dimension reduction argument by Federer in
order to prove that the dimension of the singular set is
zero. To this end, we first provide some properties of tangent maps of
a variationally biharmonic map $u\in W^{2,2}(B^+_2,N)$ in a boundary
point $a\in T_1$. For $r\in(0,1)$ we define the rescaled  maps
\begin{equation*}
  u_{a,r}\in W^{2,2}(B_{1/r}^+\,,N)\mbox{\ \ \ by\ \ \ }u_{a,r}(x):=u(a+rx).
\end{equation*}
A map $v\in W^{2,2}_{\mathrm{loc}}(\R^m_+,N)$ is called a {\em tangent map} 
of $u$ at the point $a\in T_1$ if
there is a sequence $r_i\searrow 0$ with $u_{a,r_i}\to v$ strongly in 
$W^{2,2}_{\mathrm{loc}}(\R^m,N)$ as $i\to\infty$. \\

The Compactness Theorem \ref{thm:compact} is crucial at this point to ensure
the strong convergence of suitable rescaled maps to a tangent map.
Moreover, we obtain the 
following structure theorem for tangent maps of biharmonic maps.
\begin{thm}\label{theo:tangent_maps}
  Assume that $u\in W^{2,2}(B_2^+,N)$ is a weakly biharmonic map that
  satisfies the stationarity condition~\eqref{stationary-bi} and
  satisfies $(u,Du)=(g,Dg)$ on $T_2$ in the sense of traces. 
  Moreover, we assume that $N$ does not carry any nontrivial 
  Paneitz-biharmonic $4$-spheres nor Paneitz-biharmonic
  $4$-halfspheres with constant boundary data.
  Then the following statements are true:
  \begin{enumerate}
  \item For any $a\in T_1$ and any sequence $r_i\searrow 0$ there is a
    tangent map $v\in W^{2,2}_{loc}(\R^m_+,N)$ so that $u_{a,r_i}$
    subconverges to $v$ in $W^{2,2}_{loc}(\R^m_+,N)$ as
    $i\to\infty$. The tangent map is again weakly biharmonic,
    satisfies \eqref{stationary-bi} and attains constant Dirichlet
    boundary values on $\partial\R^m_+$.
  \item  Every tangent
    map of $u$ in a point $a\in T_1$ is homogeneous of degree zero.
  \item Let $s\ge 0$. At
    $\mathcal{H}^{s}$-a.e. point $a\in\sing(u)\cap T_1$ there exists a
    tangent map $v$ of $u$ with
    $\mathcal{H}^{s}(\sing(v)\cap T_1)>0$.
  \end{enumerate}
\end{thm}
\begin{proof}
  The strong subconvergence $u_{a,r_i}\to v$ to some 
  tangent map $v\in W^{2,2}_{\mathrm{loc}}(\R^m_+,N)$ 
  is a consequence of the Compactness Theorem \ref{thm:compact}.
  The differential equations \eqref{weakly-bi} and
  \eqref{stationary-bi} are preserved under strong convergence, which
  yields the statements on the biharmonicity of the tangent
  map. The boundary values of the rescaled maps $u_{a,r_i}$ are given
  by the maps $g_{a,r_i}$, which satisfy
  \begin{equation*}
    \sup_{B_{1/r}^+}\big(|D^2g_{a,r_i}|+|Dg_{a,r_i}|\big)
    \le
    r_i\|g\|_{C^2}\to0,
    \qquad\mbox{as }i\to\infty.    
  \end{equation*}
  This implies $g_{a,r_i}\to c$ for a constant $c\in N$ and proves
  that the tangent map attains constant boundary values. This
  completes the proof of (i).

  Our next goal is to show (ii). To this end, let  
  $v$ be a tangent map of $u$ in a point $a\in T_1$. 
  Since $v$ possesses constant boundary values, the 
  monotonicity formula from Theorem~\ref{thm:monotonicity} with center
  $a=0$ reads
  \begin{equation}\label{v_monoton}
    4\int_{B_r^+\setminus B_\rho^+}\left(\frac{|D\,\partial_X
  v|^2}{|x|^{m-2}}+(m-2)\frac{|\partial_X v|^2}{|x|^m}\right)\d x
  \le
  \Phi_v(0,r)-\Phi_v(0,\rho)
  \end{equation}
  for a.e. $0<\rho<r$, where $X(x)=x$.
  The strong convergence $u_{a,r_i}\to v$ implies that the functions
  $f_i(\rho):=\Phi_{u_{a,r_i}}(0,\rho)$ converge in
  $L^1_{\mathrm{loc}}([0,\infty))$ to $f(\rho):=\Phi_v(0,\rho)$, as
  $i\to\infty$. Therefore, for any $R>0$ we deduce 
  \begin{align*}
    \mint_{R/2}^R \Phi_v(0,\rho)\,\d\rho
    &=
    \lim_{i\to\infty}\,\mint_{R/2}^R\Phi_{u_{a,r_i}}(0,\rho)\,\d\rho\\
    &=
    \lim_{i\to\infty}\,\mint_{R/2}^R\Phi_{u}(a,\rho r_i)\,\d\rho
    =
    \lim_{S\downarrow0}\,\mint_{S/2}^S\Phi_{u}(a,\sigma)\,\d\sigma,
  \end{align*}
  since the last limit exists by Corollary \ref{cor:density}.
  We deduce that the left-hand side does not depend on
  $R>0$. Therefore, inequality~\eqref{v_monoton} implies
  \begin{align*}
    \int_{B_{R/2}^+}\left(\frac{|D\,\partial_X
        v|^2}{|x|^{m-2}}+(m-2)\frac{|\partial_X v|^2}{|x|^m}\right)\d x
    =0
  \end{align*}
  for any $R>0$, which is only possible if $\partial_Xv=0$. This
  yields the asserted homogeneity of the tangent map.
  
  For the proof of the last statement (iii), 
  we recall a standard result on densities of sets, which states that
  for $\mathcal{H}^s$-a.e. point $a\in\sing(u)\cap T_1$ there is
  a sequence $r_i\searrow 0$ with
  \begin{equation}\label{pos_dens}
    \lim_{i\to\infty}r_i^{-s}\mathcal{H}^s(\sing(u)\cap T_{r_i}(a))\ge 2^{-s}\alpha(s),
  \end{equation}
  where $\alpha(s):=\Gamma(\frac12)^s/\Gamma(\frac s2+1)$, cf. Remark~3.7 in
  \cite{Simon}.
  As shown above, after passing to a subsequence, the rescaled maps $u_{a,r_i}$ converge to a tangent map 
  $v\in W^{2,2}_{\mathrm{loc}}(\R^m_+,N)$ locally with respect to the
  $W^{2,2}$-norm.
  Let us assume for contradiction that $\mathcal{H}^s(\sing(v)\cap T_1)=0$.
  Then for any $\eps>0$ there is a cover 
  $\cup_{j\in\N}B_j\supset\sing(v)\cap T_1$ of open balls $B_j$ with radii $\rho_j$ so that
  \begin{equation}
    \label{eps_cover}
    \sum_{j\in\N}\rho_j^s<\eps.
  \end{equation}
  We claim that $\sing(u_{a,r_i})\cap T_1\subset U_\eps:=\cup_jB_j$ for all but finitely many 
  values of $i\in\N$. If this was not the case, after 
  passing to a subsequence we could choose singular boundary points 
  $p_i\in\sing(u_{a,r_i})\cap T_1$ with $p_i\to p\not\in\sing(v)$ as $i\to\infty$. 
  Since $v$ is smooth in a neighbourhood of $p$, we have 
  \begin{equation*}
    r^{4-m}\int_{B_r^+(p)}\big(|\Delta v|^2+r^{-2}|Dv|^2\big)\d x+r<2^{2-m}\eps_1
  \end{equation*}
  for all sufficiently small radii $r\in(0,\delta)$, where $\eps_1>0$
  denotes the constant from the Regularity Theorem~\ref{epsreg} for
  the case of constant boundary values. 
  The preceding inequality implies for sufficiently large values of $i>i_0(r)$ that
    \begin{eqnarray*}
     \left(\tfrac
       r2\right)^{4-m}\int_{B_{r/2}^+(p_i)}\big(|\Delta(u_{a,r_i})|^2+\left(\tfrac
      r2\right)^{-2}|D(u_{a,r_i})|^2\big)\d x
      +\tfrac r2<\eps_1     
    \end{eqnarray*}
    holds true, since we can in particular assume
    $B_{r/2}^+(p_i)\subset B_r^+(p)$. Theorem~\ref{epsreg} therefore
    implies $p_i\not\in\sing(u_{a,r_i})$, which is a contradiction.  
    We conclude that $\sing(u_{a,r_i})\cap T_1\subset U_\eps$ holds true for all sufficiently large
    $i\in\N$. But this means
  \begin{equation*}
    r_i^{-s}\mathcal{H}^s(\sing(u)\cap T_{r_i}(a))=\mathcal{H}^s(\sing(u_{a,r_i})\cap T_1)
                   \le \alpha(s)\sum_{j\in\N}\rho_j^s<\alpha(s)\eps
  \end{equation*}
  for all sufficiently large $i\in\N$, where we used \eqref{eps_cover}
  in the second last step. By choosing $\eps:=2^{-s-1}$, we
  achieve a contradiction to the density condition 
  \eqref{pos_dens}, which completes the proof of the theorem.
\end{proof}

\subsection{Proof of Theorem~\ref{thm:boundary-regularity}}
Having established the properties of the tangent maps gathered in the preceding
theorem, the proof of Theorem~\ref{thm:boundary-regularity} follows in a standard way from the 
dimension reduction principle due to Federer, see for example 
\cite{Simon}, Thm. A.4. We briefly sketch the remainder of the proof.

Let $d\in\N_0$ be the smallest natural number with the property that
\begin{equation}\label{def:d}
  \Hdim(\sing u\cap T_1)\le d
\end{equation}
holds true for every weakly biharmonic map $u\in W^{2,2}(B_1^+,N)$
with~\eqref{stationary-bi} and smooth boundary values $g\in
C^\infty(B_1^+,N)$ on $T_1$. Here, $\Hdim$ denotes the Hausdorff
dimension.
We assume for contradiction that there are weakly biharmonic maps
with~\eqref{stationary-bi} and smooth boundary values that have
singular boundary points. Then by definition of $d$, we can find such a weakly
biharmonic map $u\in W^{2,2}(B_1^+,N)$ and some $s\in(d-1,d]$
with $\mathcal{H}^s(\sing u\cap T_1)>0$. By Theorem
\ref{theo:tangent_maps} there exists a tangent map $v_0$ of $u$ in a
singular boundary point $a\in \sing u\cap T_1$ that satisfies
$\mathcal{H}^s(\sing v_0\cap T_1)>0$ as well. Moreover, $v_0$ possesses
constant boundary values and is homogeneous of degree
zero. Unfortunately, we do not know whether $v_0$ is smooth in a
neighbourhood of $\partial\R^m_+\setminus\{0\}$, so that
Theorem~\ref{thm:Liouville} is not applicable. Therefore, in the
case $d>0$, we repeat the procedure and construct in turn a tangent map $v_1$ of $v_0$ in a
point $b\in \sing v_0\cap T_1\setminus\{0\}$ that again satisfies
$\mathcal{H}^s(\sing v_1\cap T_1)>0$. Moreover, this second tangent
map is homogeneous of degree zero and satisfies moreover
$\partial_bv_1\equiv0$. This implies in particular that the singular set
contains the one-dimensional subspace $\R
b\subset\partial\R^m_+$. We repeat the last step $d$ times to
successively construct tangent maps $v_1,v_2,\ldots,v_d$ whose
singular sets contain subspaces of increasingly higher dimension. The result
is a tangent map $v_d\in W^{2,2}(\R^m_+,N)$ whose singular set contains
a $d$-dimensional linear subspace $V\subset\partial\R^m_+$. After a
rotation, we may assume that $\sing v_d\supset \R^d\times\{0\}$. Moreover,
by construction, the map $v_d$ is weakly biharmonic, satisfies
\eqref{stationary-bi}, has constant boundary values on $T_1$,
is homogeneous of degree zero and constant in direction of $\R^d\times\{0\}$.
By the definition of $d$ according to \eqref{def:d}, we know
\begin{equation*}
  \sing v_d\cap \partial\R^m_+= \R^d\times\{0\}.
\end{equation*}
In view of the partial regularity result
from Corollary \ref{cor:partial}, 
this is only possible if $d\le m-5$.
We consider the restriction $\tilde
v:=v_d|_{\{0\}\times \R^{m-d}_+}\in W^{2,2}_{\mathrm{loc}}(\R^{m-d}_+,N)$.
A standard cut-off argument implies that $\tilde
v$ is stationary biharmonic on $\R^{m-d}_+$, cf. \cite[Lemma 4.3]{Scheven1}. Moreover, $\tilde v$ has
constant boundary values in the sense of $(u,Du)=(c,0)$ on
$\partial\R^{m-d}_+$, and $\tilde v$ is homogeneous of degree
zero. Moreover, by construction we have
$\sing \tilde v\cap \partial\R^{m-d}_+=\{0\}$, which means
that $\tilde v$ is smooth on a neighbourhood of
$\partial\R^{m-d}_+\setminus\{0\}$.
In the case $d\le m-6$, we may therefore apply Theorem \ref{thm:Liouville},
which implies that $\tilde v$ is constant. If $d=m-5$, the partial regularity
result from Corollary~\ref{cor:partial} and the homogeneity of $\tilde v$
imply that $\tilde v$ is smooth on $\R^{m-d}_+$. Therefore, 
our assumptions and Lemma~\ref{lem:liouville-4} again yield that $\tilde v$
is a constant map.
In any case, we obtain a contradiction to $\sing(\tilde v)\cap\partial\R^{m-d}_+=\{0\}$.
We conclude that our assumption must have been false, i.e. $\sing
u\cap T_1=\varnothing$ holds for any weakly biharmonic map $u\in
W^{2,2}(B_1^+,N)$ with~\eqref{stationary-bi} and smooth boundary
values. This completes the proof of
Theorem~\ref{thm:boundary-regularity}. 
\hfill\qed


\begin{thebibliography}{99}
\bibitem{Altuntas}S. Altuntas,
  A boundary monotonicity inequality for variationally
  biharmonic maps and applications to regularity theory.
  \emph{Ann. Global Anal. Geom.} 54(4):489--508, 2018.

\bibitem{Angelsberg} G. Angelsberg.
  A monotonicity formula for stationary biharmonic maps.
  \emph{Math. Z.} 252(2):287--293, 2006. 
  
\bibitem{Bethuel}F. Bethuel, 
On the singular set of stationary harmonic maps.
\emph{Manuscripta Math.} 78:417--443, 1993.


\bibitem{Chang}A. Chang, P. Yang.
  \emph{On a fourth order curvature invariant},
  Spectral problems in geometry and arithmetic,
  Contemp. Math. 237, pp. 9-28, Amer. Math. Soc., Providence, 1999.
  
\bibitem{CWY}A. Chang, L. Wang, P. Yang,
  A regularity theory of biharmonic maps.
  \emph{Comm. Pure Appl. Math.} 52(9):1113--1137, 1999.

\bibitem{TwoReports}  
  J.~Eells, L.~Lemaire,
  Two reports on harmonic maps.
  World Scientific Publishing Co., Inc., River Edge, NJ, 1995. 
  
\bibitem{GongLammWang}H. Gong, T. Lamm, C. Wang,
  Boundary partial regularity for a class of biharmonic maps.
  \emph{Calc. Var. Partial Differential Equations} 45(1-2):165--191, 2012. 

\bibitem{LammWang} T. Lamm, C. Wang,
  Boundary regularity for polyharmonic maps in the critical dimension.
  \emph{Adv. Calc. Var.} 2(1):1--16, 2009.  

\bibitem{Lemaire}  
  L. Lemaire,
  Applications harmoniques de surfaces riemanniennes.
  \emph{J. Differential Geom.} 13(1):51--78, 1978. 
  
\bibitem{Lin}F.H. Lin,
  Gradient estimates and blow-up analysis for stationary harmonic
  maps.
  \emph{Annals of Math.} 149:785--829, 1999.
  
\bibitem{Luckhaus}
S. Luckhaus,
Partial H\"older continuity for minima of certain energies among
  maps into a Riemannian manifold.
\emph{ Indiana Univ. Math. J.} 37(2):349--367, 1988.


\bibitem{Mattila}P. Mattila,
\emph{  Geometry of Sets and Measures in Euclidean Spaces - Fractals and
  rectifiability}.
  Cambridge Studies in advanced mathematics 44, Cambridge University Press, 1995.
  
\bibitem{Mazowiecka}K.~Mazowiecka
  Boundary regularity for minimizing biharmonic maps.
  \emph{Calc. Var. Partial Differential Equations} 57(6):Art. 143, 2018.
  
\bibitem{Moser1}R. Moser, 
  The blow-up behavior of the biharmonic map heat flow in four
  dimensions.
  \emph{Int. Math. Res. Pap.} 2005:351--402, 2005.
  
\bibitem{Moser}R. Moser, 
  A variational problem pertaining to biharmonic maps.
  \emph{Comm. Partial Differential Equations} 33(7-9):1654--1689, 2008. 

\bibitem{Paneitz}
  S. Paneitz, 
  A quartic conformally covariant differential operator for arbitrary
  pseudo-Riemannian manifolds.
  \emph{SIGMA Symmetry Integrability Geom. Methods Appl.} 4, paper
  036, 2008.
  
\bibitem{Riviere-discont}
T. Rivi\`ere, 
Everywhere discontinuous harmonic maps into spheres.
\emph{Acta Math.} 175(2):197--226, 1995.

  
\bibitem{Riviere}
T. Rivi\`ere, 
Conservation laws for conformally invariant variational problems.
\emph{Invent. Math.} 168(1):1--22, 2007. 

\bibitem{Partial-Riviere}
T. Rivi\`ere, M. Struwe, 
Partial regularity for harmonic maps and related problems.
\emph{Comm. Pure Appl. Math.} 61(4):451--463, 2008.
  
\bibitem{Scheven0}C. Scheven,  
  Variationally harmonic maps with general boundary conditions:
  boundary regularity.
  \emph{Calc. Var. Partial Differential Equations} 25(4):409--429, 2006. 
  
\bibitem{Scheven1}C. Scheven,
  Dimension reduction for the singular set of biharmonic maps.
  \emph{Adv. Calc. Var.} 1:53--91, 2008.
  
\bibitem{Scheven2}C. Scheven,
  An optimal partial regularity result for minimizers of an
  intrinsically defined second-order functional.
  \emph{Ann. Inst. Henri Poincar\'e, Anal. Non Lin\'eaire.}
  26(5):1585--1605, 2009.

\bibitem{Schoen}R. Schoen,
  \emph{Analytic aspects of the harmonic map problem}, in: S.S. Chern(ed.), Seminar on Nonlinear Partial
  Differential Equations, Springer, 1984, 321--358.

\bibitem{Schoen-U}
  R. Schoen and K. Uhlenbeck,
  A regularity theory for harmonic maps.
  \emph{J. Differential Geom.} 17(2):307--335, 1982.

\bibitem{Schoen-U2}
  R. Schoen and K. Uhlenbeck,
  Boundary regularity and the Dirichlet problem for harmonic maps.
  \emph{J. Differential Geom.} 18(2):253--268, 1983.
  
\bibitem{Simon}
  L. Simon, {\em Lectures on Geometric Measure Theory}, Proc. of
  Centre for Math. Anal., Australian National University, Vol. 3, Canberra, 1983.
    
\bibitem{Struwe}M. Struwe,
  Partial regularity for biharmonic maps, revisited.
  \emph{Calc. Var. Partial Differential Equations}
  33(2):249--262, 2008.   

\bibitem{Wang_sphere}C. Wang,
  Remarks on biharmonic maps into spheres.
  \emph{Calc. Var. Partial Differential Equations} 21(3):221--242, 2004.
  
\bibitem{Wang4d}C. Wang,
  Biharmonic maps from $\R^4$ into a Riemannian manifold.
  \emph{Math. Z.} 247(1):65--87, 2004.
  
\bibitem{Wang}C. Wang,
  Stationary biharmonic maps from $\R^m$ into a Riemannian manifold.
  \emph{Comm. Pure Appl. Math.} 57(4):419--444, 2004.
  
\bibitem{Ziemer}W. Ziemer,
  \emph{Weakly Differentiable Functions. Sobolev Spaces and Functions
    of Bounded Variation},
  Graduate Texts in Mathematics, 120, Springer, New York, 1989.
  
\end{thebibliography}
\end{document}